\documentclass[11pt]{amsart}

\usepackage{amsmath,amscd,amssymb}

\newtheorem{theorem}{Theorem}[section]

\newtheorem{proposition}[theorem]{Proposition}

\newtheorem{lemma}[theorem]{Lemma}
\theoremstyle{remark}
\newtheorem{remark}[theorem]{Remark}
\newtheorem{definition}[theorem]{Definition}
\newtheorem{remarks}[theorem]{Remarks}
\newtheorem{example}[theorem]{Example}

\newcommand\A{\mathcal{A}}
\newcommand\sA{\mathsf{A}}
\newcommand\sS{\mathsf{S}}

\newcommand\B{\mathcal{B}}
\newcommand\be{\begin{equation}\label}
\newcommand\ee{\end{equation}}

\newcommand{\K}{\mathbb{K}}

\renewcommand{\O}{\mathcal{O}}
\newcommand{\Co}{\mathcal{C}}

\newcommand{\U}{\on{U}}

\newcommand{\E}{\mathcal{E}}

\newcommand{\R}{\mathbb{R}}
\newcommand{\C}{\mathbb{C}}

\newcommand{\Z}{\mathbb{Z}}

\renewcommand{\P}{\mathcal{P}}

\newcommand{\pr}{\on{pr}}
\newcommand{\Cl}{\on{Cl}}
\newcommand\lie[1]{\mathfrak{#1}}

\newcommand{\h}{\lie{h}}
\newcommand{\g}{\lie{g}}

\newcommand{\n}{\lie{n}}

\renewcommand{\t}{\lie{t}}
\newcommand{\Alc}{\Delta}

\newcommand{\on}{\operatorname}
\newcommand{\length}{\on{length}} 
\newcommand{\Aut}{ \on{Aut} }
 
\newcommand{\Ad}{ \on{Ad} }

 \newcommand{\Spin}{ \on{Spin}}
\newcommand{\SU}{ \on{SU}} 
\newcommand{\SO}{ \on{SO}}

\newcommand{\Mult}{ \on{Mult}}

\newcommand{\Sk}{\on{Sk}}

\newcommand{\cone}{ \on{cone} }

\newcommand\dirac{/\kern-1.2ex\partial} 
\newcommand\qu{/\kern-.7ex/} 
\newcommand{\Waff}{W_{\on{aff}}} 

\newcommand{\lra}{\longrightarrow}
\newcommand{\hra}{\hookrightarrow}
\newcommand{\xra}{\xrightarrow}

\renewcommand{\d}{{\mbox{d}}}
\newcommand{\ol}{\overline}

\newcommand\sig{\sigma}
\newcommand\eps{\epsilon}
\newcommand\Om{\Omega}
\newcommand\om{\omega}

\newcommand{\f}{\frac}

\renewcommand{\H}{\ca{H}}

\renewcommand{\l}{\langle}
\renewcommand{\r}{\rangle}

\newcommand{\ti}{\tilde}

\newcommand\pt{\on{pt}}

\newcommand\beqn{\begin{equation}}      
\newcommand\eeqn{\end{equation}}      
\newcommand{\ca}{\mathcal}
\newcommand{\wh}{\widehat}
\newcommand{\wt}{\widetilde}
\newcommand{\mf}{\mathfrak}
\newcommand{\beq}{\begin{eqnarray*}}
\newcommand{\eeq}{\end{eqnarray*}}

\newcommand{\tpi}{{2\pi\sqrt{-1}}}

\newcommand{\cox}{{\mathsf{h}^\vee}}

\begin{document}

\title[]{On the quantization of 
conjugacy classes}

\vskip.2in
\author{E. Meinrenken}
\address{University of Toronto, Department of Mathematics,
40 St George Street, Toronto, Ontario M4S2E4, Canada }
\email{mein@math.toronto.edu}
\date{\today}
\begin{abstract}
  Let $G$ be a compact, simple, simply connected Lie group. A theorem
  of Freed-Hopkins-Teleman identifies the level $k\ge 0$ fusion ring
  $R_k(G)$ of $G$ with the twisted equivariant $K$-homology at level
  $k+\cox$, where $\cox$ is the dual Coxeter number of $G$. In this
  paper, we will review this result using the language of
  Dixmier-Douady bundles. We show that the additive generators of the
  group $R_k(G)$ are obtained as $K$-homology push-forwards of the
  fundamental classes of pre-quantized conjugacy classes in $G$.
\end{abstract}
\maketitle 
\setcounter{tocdepth}{2}

\section{Introduction}
A classical result of Dixmier-Douady \cite{di:ch} states that the
integral degree three cohomology group $H^3(X)$ of a space $X$
classifies bundles of $C^*$-algebras $\A\to X$, with typical fiber the
compact operators on a Hilbert space. For any such
\emph{Dixmier-Douady bundle} $\A\to X$, one defines the twisted
$K$-homology and $K$-cohomology groups of $X$ as the $K$-groups of the
$C^*$-algebra of sections of $\A$, vanishing at infinity:
\[ K_q(X,\A):=K^q(\Gamma_0(X,\A)),\ \ \ 
K^q(X,\A):=K_q(\Gamma_0(X,\A)).\] 
If a group $G$ acts by automorphisms of $\A$, one has
definitions of $G$-equivariant $K$-groups.

The twisted $K$-groups have attracted a lot of interest in recent
years, mainly due to their applications in string theory. For the case
of torsion twistings, they were pioneered by Donovan-Karoubi
\cite{don:gr} in 1963, while the general case was developed by
Rosenberg \cite{ros:co} in 1989. Rosenberg also gave an alternative
characterization of $K^0(X,\A)$ as homotopy classes of sections of a
bundle of Fredholm operators; this viewpoint was further explored by
Atiyah-Segal \cite{at:twi} (see \cite{bou:tw,tu:twi} for alternative
approaches).

One of the most natural examples of an integral degree three
cohomology class comes from Lie theory. Let $G$ be a compact, simple,
simply connected Lie group, acting on itself by conjugation. The
generator of $H^3_G(G)=\Z$ is realized by a $G$-Dixmier-Douady bundle
$\A\to G$. Let $\cox$ be the dual Coxeter number of $G$, and $k\ge
0$ a non-negative integer (the \emph{level}).  A beautiful result of
Freed-Hopkins-Teleman \cite{fr:lo2,fr:lo3} (see also
\cite{fr:tw1,fr:tw,fr:lo}) asserts that the twisted equivariant
$K$-homology at the shifted level $k+\cox$ coincides with the level
$k$ fusion ring (Verlinde algebra) of $G$:
\begin{equation}\label{eq:fht}
  K_0^G(G,\A^{k+\cox})=R_k(G).\end{equation}
Here $R_k(G)$ may be defined as the ring of positive energy level $k$
representations of the loop group $LG$, or equivalently as the
quotient $R_k(G)=R(G)/I_k(G)$ of the usual representation ring by the
level $k$ \emph{fusion ideal}.  The quotient map $R(G)\to R_k(G)$ is
realized on the $K$-homology side as push-forward under inclusion
$\{e\}\hra G$, while the product on $R_k(G)$ is given by push-forward
under group multiplication.

As a $\Z$-module, the fusion ring $R_k(G)$ is freely generated by the
set $\Lambda^*_k$ of \emph{level $k$ weights} of $G$. In this paper
the isomorphism $R_k(G)=\Z[\Lambda^*_k]$ is realized as follows. Given
$\mu\in\Lambda^*_k\subset\t^*$, (where $\t$ is the Lie algebra of a
maximal torus), let $\Co$ be the conjugacy class of the element
$\exp(\mu/k)\in G$, where the basic inner product is used to identify
$\t^*\cong\t$. We will show that there is a canonical stable
isomorphism between of the restriction $\A^{k+\cox}|_\Co$ with the
Clifford algebra bundle $\Cl(T\Co)$. This then defines a push-forward
map in twisted $K$-homology, and the image of the $K$-homology
fundamental class $[\Co]\in K_0^G(\Co,\Cl(T\Co))$ under the
push-forward is exactly the generator of $R_k(G)$ labeled by
$\mu$. This is parallel to the fact that the generators of
$R(G)=\Z[\Lambda^*_+]$ are obtained by geometric quantization of the
coadjoint orbits through dominant weights. In fact, as shown in
\cite{fr:lo3} the generators of $R_k(G)$ can also be obtained by
geometric quantization of coadjoint orbits of the loop group of $G$.
Hence, our modest observation is that this can also be carried
out in finite-dimensional terms. In a forthcoming paper with
A. Alekseev \cite{al:onq}, we will discuss more generally the
quantization of group-valued moment maps \cite{al:mom} along similar
lines.

A second theme in this paper is the construction of a 
canonical resolution of $R_k(G)$ in the category of $R(G)$-modules,
\begin{equation}\label{eq:resolution}
 0\lra C_l\xra{\partial} C_{l-1}\xra{\partial} \cdots 
\xra{\partial}C_0\xra{\epsilon} R_k(G)\lra 0
\end{equation}
where $l=\on{rank}(G)$. In more detail, let $\{0,\ldots,l\}$ label the
vertices of the extended Dynkin diagram of $G$. For each non-empty subset
$I\subset \{0,\ldots,l\}$, let $G_I\subset G$ be the maximal rank
subgroup whose Dynkin diagram is obtained by deleting the vertices
labeled by $I$. These groups have canonical central extensions $1\to
\U(1)\to \wh{G}_I\to G_I\to 1$ (described below). Let $R(\wh{G}_I)_k$
denote the Grothendieck group of all $\wh{G}_I$-representations for
which the central circle acts with weight $k$. Define
\begin{equation}
 C_p=\bigoplus_{|I|=p+1}R(\wh{G}_I)_k.\end{equation}
The differentials $\partial$ in \eqref{eq:resolution} are given by
holomorphic induction maps relative to the inclusions
$\wh{G}_I\hra\wh{G}_J$ for $J\subset I$.  As we will explain, the
chain complex $(C_\bullet,\partial)$ arises as the $E^1$-term of a
spectral sequence computing $K_\bullet^G(G,\A^{k+\cox})$, and the
exactness of \eqref{eq:resolution} implies that the spectral sequence
collapses at the $E^2$-term. Since $R_k(G)$ is free Abelian, there are
no extension problems, and one recovers the equality
$K_0^G(G,\A^{k+\cox})=R_k(G)$ as $R(G)$-modules, and hence also as
rings.  \vskip.2in

This article does not make great claims of originality. In particular,
I learned that a very similar computation of the twisted equivariant
$K$-groups of a Lie group had appeared in the article \emph{ Thom
Prospectra for loop group representations} by Kitchloo-Morava
\cite{kit:th}. The argument itself may be viewed as a natural
generalization of the Mayer-Vietoris calculation for $G=\SU(2)$, as
explained by Dan Freed in \cite{fr:tw}.  Independently, the chain
complex had been obtained by Christopher Douglas (unpublished), who
used it to obtain information about the algebraic structure of the
fusion ring $R_k(G)$.
\vskip.2in

\noindent {\bf Acknowledgments.}  I would like to thank Nigel Higson,
John Roe and Jonathan Rosenberg for help with some aspects of analytic
$K$-homology, and Nitu Kitchloo for his patient explanations of
\cite{kit:th}. I also thank Christopher Douglas and Dan Freed for very
helpful discussions.

{\small \tableofcontents \pagestyle{headings}}

\section{Review of twisted equivariant $K$-homology}
Throughout this paper, all Hilbert spaces $\H$ will be taken to be
separable, but not necessarily infinite-dimensional.  All
(topological) spaces $X$ will be assumed to allow the structure of a
countable CW-complex (respectively $G$-CW complex, in the equivariant
case).
%

\subsection{Dixmier-Douady bundles}\label{subsec:az}
\cite{di:ch,rae:mo,ros:co}
For any Hilbert space $\H$, we denote by $\on{U}(\H)$ the unitary
group, with the strong operator topology.  Let $\K(\H)$ be the
$C^*$-algebra of compact operators, that is, the norm closure of the
finite rank operators.  The conjugation action of the unitary group on
$\K(\H)$ descends to the projective unitary group, and provides an
isomorphism, $\on{Aut}(\K(\H))=\on{PU}(\H)$. A \emph{Dixmier-Douady
bundle} $\A\to X$ is a locally trivial bundle of $C^*$-algebras, with
typical fiber $\K(\H)$ and structure group $\on{PU}(\H)$, for some
Hilbert space $\H$. That is,
\begin{equation}\label{eq:aass} 
\A=\ca{P}\times_{\on{PU}(\H)}\K(\H)\end{equation}
for a principal $\on{PU}(\H)$-bundle $\ca{P}\to X$. Dixmier-Douady
bundles of finite rank are also known as Azumaya bundles
\cite{mat:fra,mat:ind}. A \emph{gauge transformation} of $\A$ is a
bundle automorphism inducing the identity on $X$, and whose
restriction to the fibers are $C^*$-algebra automorphisms.
Equivalently, the group of gauge transformations consists of sections
of the associated group bundle, $\on{Aut}(\A)=
\ca{P}\times_{\on{PU}(\H)}\on{Aut}(\K(\H))$. This group bundle has a
central extension
\begin{equation}\label{eq:ext}
 1\to X\times \U(1)\to \wt{\Aut}(\A) \to \Aut(\A)\to 1,\end{equation}
where $\wt{\Aut}(\A)=\P\times_{\on{PU(\H)}}\on{U}(\H)$. 

If $\A_1,\A_2$ are Dixmier-Douady bundles modeled on
$\K(\H_1),\K(\H_2)$, then their (fiberwise) $C^*$-tensor product
$\A_1\otimes\A_2$ is a Dixmier-Douady bundle modeled on $\K(\H_1\otimes
\H_2)$.  Also, the (fiberwise) opposite $\A^{\on{opp}}$ of a
Dixmier-Douady bundle modeled on $\K(\H)$ is a Dixmier-Douady bundle
modeled on $\K(\H^{\on{opp}})$.  Here the Hilbert space
$\H^{\on{opp}}$ is equal to $\H$ as an additive group, but with the
new scalar multiplication by $z\in \C$ equal to the old scalar
multiplication by $\ol{z}$.

A \emph{Morita trivialization} of $\A\to X$ is a pair $(\E,\psi)$,
consisting of a Hilbert space bundle $\E\to X$ and an isomorphism
$\psi\colon \K(\E)\to \A$.  A \emph{Morita isomorphism} from $\A\to X$
to $\B\to X$ is a Morita trivalization $(\ca{E},\psi)$ of
$\B\otimes\A^{\on{opp}}$. We will write $\A\simeq \B$ to indicate a 
Morita isomorphism. 

Morita isomorphism is an equivalence relation: For instance, given
$\A$ as in \eqref{eq:aass}, the bundle of `Hilbert-Schmidt operators'
\begin{equation}\label{eq:ea}
\A_{HS}:=\ca{P}\times_{\on{PU}(\H)}(\H\otimes
\H^{\on{opp}})\end{equation}
defines a Morita isomorphism $\A\simeq \A$.  A Morita trivialization
of $\A$ amounts to a lift of the structure group from $\on{PU}(\H)$ to
$\on{U}(\H)$.  The obstruction to the existence of such a lift is
given by the Dixmier-Douady class\footnote{We take all cohomology
groups with integer coefficients, unless indicated otherwise.}
\cite{di:ch,rae:mo}
\[ \on{DD}(\A)\in H^3(X).\]
The Dixmier-Douady class satisfies $\on{DD}(\A_1\otimes\A_2)
=\on{DD}(\A_1)+\on{DD}(\A_2)$ and
$\on{DD}(\A^{\on{opp}})=-\on{DD}(\A)$.  
It hence gives an isomorphism of $H^3(X)$ with the Morita isomorphism
classes of Dixmier-Douady bundles.

\begin{example}
  Let $V\to X$ be an oriented Euclidean vector bundle of rank $k$, and
  let $\Cl(V)\to X$ be the complex Clifford algebra bundle. If $k$ is
  even, the bundle $\Cl(V)$ is a Dixmier-Douady bundle. A Morita
  trivialization of $\Cl(V)$ is equivalent to the choice of a spinor
  module $\sS\to X$, which in turn is equivalent to the choice of a
  $\Spin_c$ structure on $V$. For details, see Plymen \cite{ply:st}.
  The Dixmier-Douady class $\on{DD}(\Cl(V))$ is the image of the
  Stiefel-Whitney class $w_2(V)\in H^2(X,\Z_2)$ under the Bockstein
  homomorphism. The fact that this class is 2-torsion may be seen
  directly, since the canonical anti-involution identifies
  $\Cl(V)^{\on{opp}}\cong \Cl(V)$. In the case of $k$ odd, the even
  part $\Cl^+(V)$ is a Dixmier-Douady bundle, and a similar discussion
  applies.
\end{example}

Two Morita trivializations $(\E,\psi)$ and $(\E',\psi')$ of $\A$ will
be called \emph{equivalent} if there exists an isomorphism of Hilbert
bundles $\E\to \E'$ intertwining $\psi,\psi'$. Thus, 
$(\E,\psi)$ and $(\E',\psi')$ are equivalent if and only if 
the Hermitian line bundle 
\[ L=\on{Hom}_\A(\E,\E')\to X\]
(fiberwise $\A$-module homomorphisms) is isomorphic to the trivial
line bundle.  Conversely, given a Hermitian line bundle $L\to X$, one
may `twist' a Morita trivialization $(\E,\psi)$ by setting
$\E'=\E\otimes L$, with $\psi'$ the composition of $\psi$ with the
natural isomorphism $\K(\E')\cong \K(\E)$. In this way, one obtains a
transitive, free action of $H^2(X)$ on isomorphism classes of Morita
trivializations of $\A$, provided $\on{DD}(\A)=0$. (In the example
$\A=\Cl(V)$, this is the usual twist of $\Spin_c$-structures by line
bundles.) More generally, $H^2(X)$ acts freely and transitively on
equivalence classes of Morita isomorphisms $\A\simeq \B$, provided
$\on{DD}(\A)=\on{DD}(\B)$.

The \emph{stabilization} of a Dixmier-Douady bundle $\A\to X$ is
defined as $\A^{\on{st}}=\A\otimes \K$, where $\K$ denotes the compact
operators on a fixed infinite-dimensional Hilbert space $\mathbb{H}$.
A \emph{stable isomorphism} between $\A,\B\to X$ is an isomorphism
$\A^{\on{st}}\xra{\cong}\B^{\on{st}}$. Two stable isomorphisms are
called equivalent if they are homotopic.  It is known (cf. Appendix A)
that equivalence classes of stable isomorphisms are in 1-1
correspondence with equivalence classes of Morita isomorphisms.

Given a compact Lie group $G$ acting on $X$, one may similarly define
$G$-equivariant Dixmier-Douady bundles. All of the above extends to
this equivariant setting: In particular, there is a $G$-equivariant
Dixmier-Douady class $\on{DD}_G(\A)\in H^3_G(X)$, which classifies
$G$-Dixmier-Douady bundles up to $G$-Morita isomorphism $(\E,\psi)$.
(That is, $\E$ is a $G$-Hilbert space bundle $\E\to X$ and $\psi\colon
\ca{B} \otimes \A^{\on{opp}} \to \K(\E)$ is a $G$-equivariant
  isomorphism.) The stabilization of a $G$-Dixmier-Douady bundle is
  defined as $\A\otimes \K_G$, where $\K_G$ is the $G-C^*$-algebra of
  compact operators on a fixed $G$-Hilbert space $\mathbb{H}_G$, containing all
  irreducible finite-dimensional unitary $G$-representations with
  infinite multiplicity. The extension of the Dixmier-Douady theorem
  to the $G$-equivariant case was proved by Atiyah-Segal
  \cite{at:twi}.

Still more generally, one can also consider $\Z_2$-graded
$G$-Dixmier-Douady bundles $\A\to X$. Here, isomorphisms and tensor
products are understood in the $\Z_2$-graded sense, and the Hilbert
space bundles in the definition of stable isomorphism are
$\Z_2$-graded.  
If $\A\to X$ is such a bundle, and $\on{DD}_G(\A)=0$ (so that there
exists a $G$-equivariant isomorphism $\K(\E)\to \A$, ignoring the
$\Z_2$-grading), there is an obstruction in $H^1(X,\Z_2)$ for the
existence of a compatible $\Z_2$-grading on $\E$. That is, the map
from Morita equivalence classes of $\Z_2$-graded $G$-Dixmier-Douady
bundles to those of ungraded $G$-Dixmier-Douady bundles (forgetting
the $\Z_2$-grading) is onto, with kernel $H^1(X,\Z_2)$. See
\cite{at:twi} for details.

\subsection{Dixmier-Douady bundles related to central extensions}
\label{subsec:trv} 
We assume that $G$ is compact and connected. Then $H^1_G(\pt)=0$,
while $H^2_G(\pt)$ is the group of $G$-equivariant line bundles over a
point, or equivalently $H^2_G(\pt)=\on{Hom}(G,\U(1))$.
Consider central extensions of $G$ by $\U(1)$, 
\begin{equation}\label{eq:extgr} 
1\to \U(1)\to \wh{G}\to G\to 1.
\end{equation}
For any central extension $\wh{G}$, there is an associated line bundle
$L=\wh{G}\times_{\U(1)}\C\to G$. Powers $\wh{G}^{(l)},\ l\in\Z$ are
obtained as unit circle bundles inside the corresponding powers $L^l$
of the associated line bundles (defined as $(L^*)^{|l|}$ for $l<0$),
similarly products of central extensions are defined in terms of
tensor products of line bundles. Thus, the central extensions
\eqref{eq:extgr} form an Abelian group, with unit the trivial
extension. The group of gauge transformations of a central extension
$\wh{G}$ (i.e. group automorphims covering the identity on $G$) is
$\on{Hom}(G,\U(1))$.

It is well-known that the group of isomorphism classes of central
extensions is isomorphic to $H^3_G(\pt)$.  From the interpretation via
Dixmier-Douady bundles, this may be seen as follows: Given a
$G$-equivariant Dixmier-Douady algebra $\A$, (viewed as a bundle over
$\pt$), one obtains a central extension of $G$ as a pull-back of
$\on{U}(\H)\to \on{Aut}(\K(\H))$ by the homomorphism $G\to \Aut(\H)$.
Conversely, given $\wh{G}$, let choose a unitary representation
$\wh{G}\to \U(\H)$ where the central circle $\U(1)$ acts by scalar
multiplication. For instance, we may take $\H$ to be the space of
$L^2$-sections of the associated line bundle $\wh{G}\times_{U(1)}\C$.
Then $\K(\H)\to \pt$ is a $G$-Dixmier-Douady bundle with the
prescribed class in $H^3_G(\pt)$.

Suppose $X$ is a connected space, with $H^1(X)$ torsion-free, and with
the trivial action of $G$. The Kuenneth theorem \cite[Chapter
5.5]{spa:al} for $H^\bullet_G(X)=H^\bullet(X\times BG)$ gives a direct
sum decomposition,
\[ H^3_G(X)=H^3(X)\oplus (H^1(X)\otimes H^2_G(\pt))\oplus H^3_G(\pt).\]
For any $G$-Dixmier-Douady bundle $\A\to X$, we obtain a corresponding
decomposition of $\on{DD}_G(\A)$. The first component is the
non-equivariant class $\on{DD}(\A)$. The last summand is the class of
the central extension of $G$, defined by the homomorphism $G\to
\on{Aut}(\A_{x_0})$ at any given base point $x_0\in X$. To describe
the second summand, note that the family of actions $G\to \on{Aut}(\A_x)$ 
defines a family of central extensions, 
\[1\to \U(1)\to \wh{G}_{(x)}\to G\to 1.\] 
For any $x'\in x$, there exists an isomorphism $\wh{G}_{(x)}\to
\wh{G}_{(x')}$ of central extensions, unique up to
$\on{Hom}(G,\U(1))\cong H^2_G(\pt)$. Since the latter group is
discrete, it follows that the family $\wh{G}_{(x)}$ carries a flat connection: Any path from a base point $x_0$ to $x$ defines an
isomorphism $\wh{G}:=\wh{G}_{(x_0)}\to \wh{G}_{(x)}$, depending only
on the homotopy class of the path. We therefore obtain a holonomy
homomorphism $\tau\colon \pi_1(X;x_0)\to H^2_G(\pt)$, hence an element
of $H^1(X)\otimes H^2_G(\pt)\subset H^3_G(X)$. This element is
identified with the corresponding component of $\on{DD}_G(\A)$.

\begin{remark}
Any element of $H^1(X)\otimes H^2_G(\pt)$ is
realized in this way. Indeed, let $\H=L^2(G)$ with the left-regular
representation of $G$. The homomorphism $\tau\colon \pi_1(X)\to
H^2_G(\pt)=\on{Hom}(G,\U(1))$ defines a unitary action of $\pi_1(X)$
on $\H$, where $\lambda\in \pi_1(X)$ acts as pointwise multiplication by the
function $\tau(\lambda)$. The actions of $G$ and $\pi_1(X)$ commute up
to a scalar. The associated bundle $\A=\wt{X}\times_{\pi_1(X)}\K(\H)$ is a
$G$-equivariant Dixmier-Douady bundle, with $\on{DD}_G(\A)$ the
prescribed class in $H^1(X)\otimes H^2_G(\pt)$. Note that the
component in $H^3(X)$ is zero, since non-equivariantly $\A=\K(\E)$ for
$\E=\wt{X}\times_{\pi_1(X)}\H$.
\end{remark}

\subsection{Twisted $K$-homology}\label{subsec:twi}
The input for the twisted equivariant $K$-homology of a $G$-space $X$
is a $\Z_2$-graded $G$-Dixmier-Douady bundle $\A\to X$.  From now
on, we we usually omit explicit mention of the $\Z_2$-grading (which
may be trivial), with the understanding that all tensor products are
in the $\Z_2$-graded sense, isomorphisms should preserve the
$\Z_2$-grading, and so on.

Given $\A\to X$, the space $\sA=\Gamma_0(X,\A)$ of continuous
sections of $\A$ vanishing at infinity is a ($\Z_2$-graded)
$G-C^*$-algebra, with norm $||s||=\sup_{x\in X} ||s_x||_{\A_x}$.
Following J. Rosenberg \cite{ros:co}, we define the twisted
equivariant $K$-homology and $K$-cohomology groups as the 
equivariant $C^*$-algebra
$K$-homology and $K$-cohomology groups of $\sA$:
\[ K_q^G(X,\A):=K^q_G(\Gamma_0(X,\A)),\ \ 
K^q_G(X,\A):=K_q^G(\Gamma_0(X,\A)).\]
In this paper, we will mostly work with the $K$-homology groups.  See
the appendix for a quick review of the $K$-homology of $C^*$-algebras,
and some examples. For instance, if $X=M$ is a compact manifold, and
$\A$ has finite rank, then every $G$-invariant first order elliptic
differential operator, acting on the sections of a $G$-equivariant
bundle of $\A$-modules, defines a twisted equivariant $K$-homology
class.  Some of the basic properties of the $K$-homology groups are as
follows.

\begin{enumerate}
\item {\bf Stability}. 
The twisted $K$-groups are unchanged under stabilization:
\[ K_q^G(X,\A)=K_q^G(X,\A\otimes\mathbb{K}_G).\]
In more detail, recall that $\K_G$ denotes the compact operators on a
fixed stable $G$-Hilbert space. Let $p\in \K_G$ be the projection
operator onto a 1-dimensional invariant subspace.  Then the map $\A\to
\A\otimes\K_G,\ a\mapsto a\otimes p$ induces an isomorphism in twisted
$K$-homology. Since $p$ is unique up to homotopy, the induced map 
in $K$-homology does not depend on the choice of $p$. 
\item {\bf Morphisms}. The morphisms in the category of
      $G$-Dixmier-Douady bundles $(X,\A)$ are the equivariant
      $C^*$-algebra bundle map $\A_1\to \A_2$ for which the induced
      map on the base $f\colon X_1\to X_2$ is \emph{proper}. Any such
      morphism induces a morphism of $G-C^*$-algebras $f^*\colon
      \Gamma_0(X_2,\A_2)\to \Gamma_0(X_1,\A_1)$, hence a push-forward
      in $K$-homology
\[ K_q^G(f)\colon K_q^G(X_1,\A_1)\to K_q^G(X_2,\A_2).\]
Then $K_\bullet^G$ becomes a covariant functor, invariant under proper
$G$-homotopies.  More generally, using (a) it suffices to have a
\emph{stable} morphism between $\A_1,\A_2$, i.e. a morphism of their
stabilizations $\A_i\otimes\K_G$.  

\item {\bf Excision}. For any closed, invariant subset $Y\subset X$,
      with complement $U=X\backslash Y$,  
there is a long exact sequence
\[ \cdots \to K_q^G(Y,\A|_Y)\to K_q^G(X,\A)\to K_q^G(U,\A|_U)\to
K_{q-1}^G(Y,\A|_Y)\to \cdots\] 
Here the restriction map $K_q^G(X,\A)\to
K_q^G(U,\A|_U)$ is induced by the $C^*$-algebra morphism
$\Gamma_0(U,\A|_U)\to \Gamma_0(X,\A)$, given as extension by $0$.
\footnote{Note that $K$-homology is analogous to Borel-Moore homology
  (homology with non-compact supports), rather than ordinary
  homology.}  More generally, one obtains a spectral sequence for any
filtration of $X$ by closed, invariant subspaces.
\item {\bf Products}. Suppose $\A\to X$ and $\B\to Y$ are two
      $G$-Dixmier-Douady bundles. Then the exterior tensor product
      $\A\boxtimes \B\to X\times Y$ is again a $G$-Dixmier-Douady
      bundle. Its space of sections is the $C^*$-tensor product of the
      spaces of sections of $\A,\B$. As a special case of the Kasparov
      product in $K$-homology, one has a natural associative cross product,
\[ K_\bullet^G(X,\A)\otimes K_\bullet^G(Y,\B)\to K_\bullet^G(X\times Y,\
\A\boxtimes\B).\]
\item{\bf Module structure}. The group $K_0^G(\pt)$ is canonically
      identified with the representation ring $R(G)$. The ring
      structure on $K_0^G(\pt)$ is defined by the cross product for
      $\C\boxtimes\C\to \pt\times\pt$.  Similarly, if $\A\to X$ is a
      $G$-Dixmier-Douady bundle, the cross product for
      $\C\boxtimes\A\to 
      \pt\times X$ makes $K^G_\bullet(X,\A)$ into a module over
      $R(G)$.  The maps $K_q^G(f)$ are $R(G)$-module homomorphisms.
\end{enumerate}
If $M$ is a manifold
one has the \emph{Poincar\'{e} duality isomorphism} relating twisted $K$-homology and $K$-cohomology, 
\begin{equation}\label{eq:poincare}
 K_{q}^G(M,\A)\cong K^{q}_G(M,\A^{\on{opp}}\otimes\Cl(TM)).\end{equation}
Here $\Cl(TM)$ is the Clifford algebra bundle for some choice of
invariant metric.  For $\A=\C$ the Poincar\'{e} duality was proved by
Kasparov in \cite[Section 8]{ka:co}; the result in the twisted case
was obtained by J.-L. Tu \cite[Theorem 3.1]{tu:twi1}. (See also
\cite[Section 2]{bro:dbr}). The image of $1\in K^0_G(M)$ under this
isomorphism is Kasparov's \emph{K-homology fundamental class}
\cite{ka:con}
\[ [M]\in K_{0}^G(M,\Cl(TM)).\]
\begin{remark}
  Note that $\Cl(TM)$ is a Dixmier-Douady bundle only if $\dim M$ is
  even. However, the definition of the twisted $K$-groups works for
  arbitrary bundles of $C^*$-algebras, and the isomorphism
  \eqref{eq:poincare} holds in this sense (but with $\A$ a
  Dixmier-Douady bundle). Alternatively, one may state the result in
  terms of Dixmier-Douady bundles, using $\Cl(TM)=\Cl^+(TM)\otimes
  \Cl(\R)$ and the isomorphism the isomorphism
  $K_{q+1}^G(M,\B)=K_q^G(M,\B\otimes\Cl(\R))$.
\end{remark}
The following basic computations in twisted equivariant $K$-homology
may be deduced from their $K$-theory counterparts, using Poincar\'{e}
duality. 
\begin{enumerate}
\item\label{it:a} If $M=\pt$, the twisted $K$-homology is
\[K^G_0(\pt,\A)=R(\wh{G})_{-1},\] 
while $K^G_1(\pt,\A)=0$. Here $\wh{G}$ is the central extension
defined by the action $G\to \Aut(\A)$, and $R(\wh{G})_{-1}$ is the
Grothendieck group of $\wh{G}$-representations where the central
$\U(1)$ acts with weight $-1$.
\item\label{it:b} Suppose $H$ is a closed subgroup of $G$. 
For any $H$-Dixmier-Douady bundle
      $\B\to Y$, there is a natural \emph{induction isomorphism},
\[\on{Ind}_H^G\colon K_{q}^H(Y,\B\otimes\Cl(\g/\h)) \xra{\cong}
K_{q}^G(G\times_H Y,\,G\times_H \B),\]
which is Poincar\'{e} dual to the isomorphism $K^q_G(G\times_H
Y,\,G\times_H \B^{\on{opp}})\cong K^q_H(Y,\B^{\on{opp}})$.  If
$Y=\pt$, the left hand side may be evaluated as in (a).  If $H\subset
H'\subset G$ are closed subgroups, we have
\[\on{Ind}_H^G=\on{Ind}_{H'}^G\circ \on{Ind}_H^{H'}.\] 
Here we are identifying $\Cl(\g/\h)\cong
\Cl(\g/\h')\otimes\Cl(\h'/\h)$, and we are using the canonical
isomorphism $H'\times_H \Cl(\g/\h')\cong H'/H\times\Cl(\g/\h')$.
\item\label{it:c} Let $\A\to\pt$ be a $G$-Dixmier-Douady algebra as in
\eqref{it:a}, and let $H$ be a closed subgroup of $G$. Then
$G\times_H\A$ is canonically isomorphic to $\pi^*\A$, the pull-back
under the map $\pi\colon G/H\to \pt$.  By composing the map
$\on{Ind}_H^G$ with the push-forward $K_q^G(\pi)$, we obtain an
\emph{induction map}
\[ \on{ind}_H^G\colon K_{q}^H(\pt,\A\otimes\Cl(\g/\h)) \to  
K_q^G(\pt,\A).\]
An $H$-invariant complex structure on $\g/\h$ defines a stable
trivialization of $\Cl(\g/\h)$; the resulting map
\[ \on{ind}_H^G\colon K_0^H(\pt,\A)=R(\wh{H})_{-1}
\to K_0^G(\pt,\A)=R(\wh{G})_{-1}
\]
is \emph{holomorphic induction} for the complex structure on 
$G/H=\wh{G}/\wh{H}$.   
\end{enumerate}
For other examples of calculations of twisted K-groups, see 
\cite[Section 8]{bou:tw}.

\section{The Dixmier-Douady bundle over $G$} 
For the rest of this paper, $G$ will denote a compact, simple, simply
connected Lie group, acting on itself by conjugation.  Then $H^3_G(G)$
is canonically isomorphic to $\Z$. Hence there exists a
$G$-Dixmier-Douady bundle $\A\to G$, unique up to Morita isomorphism,
such that $\on{DD}_G(G,\A)$ corresponds to the generator $1\in\Z$.
Any two bundles $\A,\A'\to G$ representing the generator are related
by a $G$-equivariant Morita isomorphism, unique up to equivalence
(since $H^2_G(G)=0$). The quickest construction of $\A$ is
as an associated bundle
\[ \A=P_eG\times_{L_eG}\K(\H),\] 
where $P_eG$ is the space of based paths in $G$, $L_eG=LG\cap P_eG$
the based loop group, and $\H$ a representation of the standard
central extension $\wh{LG}$ of $LG$ where the central circle acts with
weight $-1$. The construction given in this Section is essentially
just a slow-paced version of this model for $\A$, avoiding
technicalities such as the choice of topology on the loop group. Our
strategy is to first give a direct construction of the family of
central extensions of the centralizers $G_g\subset G$, corresponding
to their action on $\A$.

\subsection{Pull-back to the maximal torus}
Let $T\subset G$ be a maximal torus of $G$, with Lie algebra $\t$. 
Consider the map 
\begin{equation}\label{eq:hom} H^3_G(G)\to H^3_T(T) \end{equation}
obtained by first restricting the action to $T$ and then pulling back to $T$.
Since the map $H^3(G)\to H^3(T)$ defined by the inclusion is just the
zero map, the image of \eqref{eq:hom} lies in the kernel of the map
$H^3_T(T)\to H^3(T)$, i.e it is contained in the subgroup
\[ H^2_T(\pt)\otimes H^1(T)\subset H^3_T(T).\]
We will compute the image of the generator of $H^3_G(G)$ under this
map.  Denote by $\Lambda\subset\t$ the integral lattice (i.e. the
kernel of $\exp\colon\t\to T$). Recall that the \emph{basic inner
  product} $B$ on the Lie algebra $\g$ is the unique invariant inner
product, with the property that the smallest length of a non-zero
element $\lambda\in \Lambda$ equals $\sqrt{2}$. One of the key
properties of $B$ is that it restricts to an integer-valued bilinear
form on $\Lambda$.  That is, $B|_\t\in\Lambda^*\otimes\Lambda^*$ where
$\Lambda^*=\on{Hom}(\Lambda,\Z)\subset\t^*$ is the (real) weight
lattice.
\begin{lemma}\label{lem:torusres}
The map \eqref{eq:hom} is injective, and takes the 
generator of $H^3_G(G)$ to the element
\begin{equation}\label{eq:Brest}
 B|_\t\in\Lambda^*\otimes\Lambda^*
\cong H^2_T(\pt)\otimes H^1(T)\end{equation}
given by the basic  inner product.
 \end{lemma}
\begin{proof}
  Since $H_G(G)$ and $H_T(T)$ have no torsion, we may pass to real
  coefficients, and hence work with Cartan's equivariant de Rham model
  $\Omega_G^p(M)=\bigoplus_{2i+j=p}(S^i\g^*\otimes\Omega^j(M))^G$ for
  the equivariant cohomology $H_G(M,\R)$ of a $G$-manifold, with
  differential
  $(\d_G\alpha)(\xi)=\d\alpha(\xi)-\iota(\xi_M)\alpha(\xi)$ where
  $\xi_M$ is the vector field defined by $\xi\in\g$.  Let
  $\theta^L,\theta^R\in\Omega^1(G,\g)$ be the left-, right- invariant
  Maurer Cartan forms.  The generator of $H^3_G(G)$ is represented by
  an equivariant de Rham form (see e.g. \cite{me:enc}),
\begin{equation}\label{eq:etag}
 \eta_G(\xi)=\f{1}{12}B(\theta^L,[\theta^L,\theta^L])
+\f{1}{2}B(\theta^L+\theta^R,\xi).\end{equation}
Its image under the map $\Om_G(G)\to \Omega_T(T)$ is
\[ \iota_T^*\eta(\xi)=B(\theta_T,\xi)\in 
\Omega^2_T(\pt)\otimes \Omega^1(T)=\t^*\otimes\Omega^1(T)\] 
with $\theta_T\in\Om^1(T,\t)$ the Maurer-Cartan form for $T$. The
inclusion $\t^*\to \Om^1(T),\ \mu\mapsto \l\mu,\theta_T\r$ induces an
isomorphism in cohomology, $\t^*\cong H^1(T,\R)$. Since
$B(\theta_T,\cdot)$ represents the image of
$B|_\t\in\t^*\otimes\t^*\subset\t^*\otimes\Om^1(T)$, the proof is
complete.
\end{proof}

As discussed in Section \ref{subsec:trv}, elements of
$H^2_T(\pt)\otimes H^1(T)$ are realized as the holonomy of a family of
central extensions. For any $\mu\in\Lambda^*$ let $T\to \U(1),\ 
t\mapsto t^\mu$ be the corresponding homomorphism.  Let the lattice 
$\Lambda$ act on $\wh{T}=T\times\U(1)$ as 
\[\Lambda\times\wh{T}\to \wh{T},\ \ (\lambda;h,z)\mapsto (h,h^{-B^\flat(\lambda)}z).\] 
Then the holonomy of the family
\begin{equation}\label{eq:central}
 \t\times_\Lambda \wh{T}\to \t/\Lambda=T.
\end{equation}
is the element $-B|_\t$. Lemma \ref{lem:torusres} implies: 

\begin{lemma}\label{lem:criterion}
  Suppose $\A\to G$ is a $G$-Dixmier-Douady bundle, where the action
  on the base $G$ is given by conjugation. Then $\on{DD}_G(\A)$
  represents the generator of $H^3_G(G)$, if and only if the family of
  central extensions of $T$, defined by the $T$-action on $\A|_T$, is
  isomorphic to the \emph{opposite} of the family \eqref{eq:central} (i.e.,
  using the action $(h,z)\mapsto (h,h^{B^\flat(\lambda)}z)$). 
\end{lemma}

Let $\wh{T}_{(t)}$ be the fiber of \eqref{eq:central} over $t\in T$.
The choice of $\xi$ with $\exp\xi=t$ defines a trivialization
\begin{equation}\label{eq:xitriv}
 T\to \wh{T}_{(t)}\subset \t\times_\Lambda \wh{T},\ h\mapsto
 [(\xi;h,1)]
\end{equation}
where the brackets indicate equivalence classes.  Shifting $\xi$ by
$\lambda\in\Lambda$ changes the trivialization by the homomorphism
$T\to \U(1),\ h\mapsto h^{-B^\flat(\lambda)}$.  There is the following
equivalent description of $\wh{T}_{(t)}$.
\begin{lemma}
For any $t\in T$ there is a canonical isomorphism
$\wh{T}_{(t)}=\t\times_\Lambda\U(1)$, where $\Lambda$ acts on
$\U(1)$ by the homomorphism 
\begin{equation} \label{eq:varrhot}
\varrho_t\colon \Lambda\to \U(1),\ \ \lambda\mapsto t^{B^\flat(\lambda)}.
\end{equation}
In terms of this identification, the trivialization of
$\wh{T}_{(t)}=\t\times_\Lambda\U(1)$ defined by the choice of $\xi\in\t,\
\exp\xi=t$ reads, 
\begin{equation}\label{eq:triviall}
 T=\t/\Lambda \mapsto \t\times_\Lambda\U(1),\ \big[\zeta\big]\mapsto 
\big[(\zeta,e^{-\tpi B(\xi,\zeta)})\big].
\end{equation}
\end{lemma}
\begin{proof}
The choice of $\xi$ with $\exp\xi=t$ defines an isomorphism
\[\t\times_\Lambda \U(1)\to
\wh{T}_{(t)},\ \big[(\zeta,z)\big]\mapsto \big[(\xi;\,\exp\zeta,\, e^{\tpi B(\xi,\zeta)}z)\big].\]
It is straightforward to check that this map is well-defined, and
independent of the choice of $\xi$.  Its composition with
\eqref{eq:triviall} is $\big[\zeta\big]\mapsto
\big[(\xi;\,\exp\zeta,\, 1)\big]$, the trivialization of $\wh{T}_{(t)}$
  defined by the choice of $\xi$.
\end{proof}

The action of the Weyl group $W=N(T)/T$ on $T$ lifts to an action on the bundle
\eqref{eq:central}, given as
\[ w.[(\xi;h,z)]=[(w\xi;wh,z)].\]
If $w$ fixes $\exp\xi$, so that $w\xi=\xi-\lambda$ for some
$\lambda\in\Lambda$, the formula may be written
\[ w.[(\xi;h,z)]=[(\xi-\lambda;wh,z)]=[(\xi;wh,h^{-B^\flat(\lambda)})].\]
This leads us to another description of the bundle \eqref{eq:central}.
Let $\t_+\subset\t$ be the choice of a closed Weyl chamber, and
let $\Alc\subset \t_+$ be the corresponding closed Weyl alcove. Recall
that $\Alc$ labels the $W$-orbits in $T$, in the sense that every
orbit contains a unique point in $\exp(\Alc)$. Label the vertices of
$\Alc$ by $0,\ldots,l=\on{rank}(G)$, in such a way that the label $0$
corresponds to the origin.  For every non-empty subset $I\subset
\{0,\ldots,l\}$ let $\Alc_I$ denote the closed simplex spanned by the
vertices in $I$, and let $W_I\subset W$ denote the subgroup fixing
$\exp(\Alc_I)\subset T$. Then the maps $W/W_I\times\Alc_I\to T,\
(wW_I,\xi)\mapsto w\exp\xi$ define an isomorphism
\begin{equation}\label{eq:torusdesc}
 T\cong \coprod_I W/W_I\times \Alc_I\Big/\sim
\end{equation}
using the identifications,
\begin{equation}\label{eq:eqrelation}
 (x,\iota_J^I(\xi))\sim (\phi_I^J(x),\xi),\ \ \ J\subset I.
\end{equation}
Here $\iota_J^I\colon \Alc_J\hra \Alc_I$ is the natural inclusion,
giving rise to an inclusion $W_I\hra W_J$ of Lie groups and hence to
projection $\phi_I^J\colon W/W_I\to W/W_J$. With similar
identifications we have, 
\[ \t\times_\Lambda\wh{T}=\coprod_I (W\times_{W_I} \wh{T})\times \Alc_I\Big/\sim.\]
Here $W_I$ acts on $\wh{T}$ by 
\begin{equation}\label{eq:whynot}
w.(h,z)=(w.h,\,h^{-B^\flat(\lambda)}z)\end{equation} 
where $\lambda\in\Lambda$ is
the unique lattice vector with $\lambda+w\Alc_I= \Alc_I$.

\subsection{The centralizers $G_I$}
Let $G_I\subset G$ be the subgroup of $G$ fixing $\exp(\Alc_I)\subset
G$. Equivalently, $G_I$ is the centralizer of any element $t=\exp\xi$
with $\xi\in\t$ in the interior of the face $\Alc_I$. Each $G_I$ is a
connected subgroup containing $T$, and we have $W_I=N_{G_I}(T)/T$.
For $J\subset I$ we have $G_I\subset G_J$. The description
\eqref{eq:torusdesc} of the maximal torus extends to the group $G$:
\begin{equation}\label{eq:Gsim}
 G\cong \coprod_I G/G_I \times \Alc_I\Big/\sim
\end{equation}
using the equivalence relations \eqref{eq:eqrelation} for the natural
maps $ \phi_I^J\colon G/G_I\to G/G_J$ for $J\subset I$.

\begin{lemma}
There are distinguished central extensions
\[ 1\to \U(1)\to \wh{G}_I\to G_I\to 1 ,\]
together with lifts $\wh{i}_I^J\colon \wh{G}_I\hra \wh{G}_J$ of the
inclusions $i_I^J\colon G_I\hra G_J$ for $J\subset I$, such that
\begin{enumerate}
\item $\wh{G}_{\{0,\ldots,l\}}=\wh{T}$, 
\item the lifted inclusions satisfy the coherence condition
      $\wh{i}_I^K =\wh{i}_J^K \circ \wh{i}_I^J$ for $K\subset J\subset
      I$,
\item the $W_I$-action on $\wh{T}\subset \wh{G}_I$ (cf.
      \eqref{eq:whynot}) is induced by the conjugation action of
      $N_{G_I}(T)$.
\end{enumerate}
\end{lemma}
\begin{proof}
Recall $\pi_1(G_I)=\Lambda/\Lambda_I$, where $\Lambda_I$ is the
co-root lattice of $G_I$ \cite[Theorem (7.1)]{br:rep}. But 
\[ \lambda\in \Lambda_I,\ \ t\in\exp(\Alc_I) \Rightarrow
t^{B^\flat(\lambda)}=1\] 
(see \cite[Proposition 5.4]{me:ge}). Hence,
for $t\in \exp(\Alc_I)$, the homomorphism $\varrho_t$ defined in
\eqref{eq:varrhot} descends to a homomorphism
\[ \varrho_{t,I}\colon \pi_1(G_I)\to \U(1).\]
We therefore obtain a family of central extensions
$\wh{G}_{I,(t)}=\wt{G_I}\times_{\pi_1(G_I)} \U(1)$ parametrized by the
points of $\exp(\Alc_I)$. Since $\exp(\Alc_I)$ is contractible, we may
use the flat connection on the family of central extensions (cf.
Section \ref{subsec:trv}) to identify all $\wh{G}_{I,(t)}$. The
resulting $\wh{G}_I$ has the desired properties. In particular, if
$J\subset I$ and $t\in \exp(\Alc_J)\subset \exp(\Alc_I)$, the
homomorphism $\varrho_{t,I}$ is given by the inclusion $\pi_1(G_I)\to
\pi_1(G_J)$ followed by $\varrho_{t,J}$. This defines an inclusion
$\wh{G}_{I,(t)}\hra \wh{G}_{J,(t)}$, compatible with the flat
connection and (hence) satisfying the coherence condition.
\end{proof}
\begin{remarks}
\begin{enumerate}
\item The inclusion of $\wh{T}=T\times\U(1)$ into $\wh{G}_I\cong \wh{G}_{I,(t)}$ is
      explicitly given as (see Equation \eqref{eq:triviall})
\begin{equation}\label{eq:standard}
(\exp_T\zeta,z) \mapsto [(\exp_{\wt{G}_I}\zeta,\, e^{-\tpi
B(\xi,\zeta)})].\end{equation}
\item The lifts $\wh{i}_I^J$ of $i_I^J\colon G_I\hra G_J$
      intertwine the inclusions $\wh{T}\hra \wh{G}_I$, and are the
      unique lifts having this property.
\item
The central extensions $\wh{G}_I$ are sometimes non-trivial, even though 
of course their Lie algebra $\wh{\g}_I$ is a trivial central extension 
of $\g_I$. The choice of any $t\in\exp(\Alc_I)$ identifies 
$\wh{\g}_I\cong \wh{\g}_{I,(t)}=\g_I\times\R$, by the 
definition of $\wh{G}_{I,(t)}$ as a quotient of $\wt{G}_I\times\U(1)$. 
\end{enumerate}
\end{remarks}

\subsection{Construction of the Dixmier-Douady bundle $\A\to G$}
Our construction of the Dixmier-Douady bundle $\A\to G$ involves a
suitable Hilbert space $\H$.
\begin{lemma}\label{lem:hilbert}
There exists a Hilbert space $\H$, equipped with unitary representations of the
central extensions $\wh{G}_I$ such that (i) the central $\U(1)$ acts
with weight $-1$, and (ii) for $J\subset I$ the action of $\wh{G}_J$
restricts to the action of $\wh{G}_I$. 
\end{lemma}
One may construct such an $\H$ using the theory of affine Lie
algebras. Let $\ca{L}(\g)=\g^\C\otimes\C[z,z^{-1}]$ be the loop
algebra associated to $\g$. For all roots $\alpha$ of $G$, let
$e_\alpha\in \g^\C$ be the corresponding root vector.  Then $\g_I^\C$
is spanned by $\t^\C$ together with the root vectors $e_\alpha$ such
that $\l\alpha,\xi\r\in\Z$ for $\xi\in\Alc_I$. The map $j_I\colon
\g_I^\C\to \ca{L}(\g)$ given by $\zeta\mapsto \zeta\otimes 1$ for
$\zeta\in\t^\C$ and
\[ e_\alpha\mapsto e_\alpha\otimes z^{\l\alpha,\xi\r},\]
for $\l\alpha,\xi\r\in\Z$ is an injective Lie algebra homomorphism
(independent of $\xi$). Consider the standard central extension
$\wh{\ca{L}}(\g)=\ca{L}(\g)\oplus \C\mf{c}$, with bracket
\[ [\zeta_1\otimes f_1+s_1\mf{c},\ \zeta_2\otimes f_2+s_2\mf{c}]=([\zeta_1,\zeta_2] \otimes
f_1 f_2) +
B(\zeta_1,\zeta_2) \on{Res}(f_1 \d f_2) \mf{c}. \]
Its restriction to constant loops is canonically trivial, thus 
$\wh{\t}^C$ is embedded in $\ca{L}(\g^\C)$ by the map 
$(\zeta,s)\mapsto \zeta+s\mf{c}$. The inclusions $j_I$ lift  
to inclusions $\wh{j}_I\colon \wh{\g}_I\hra \wh{\ca{L}}(\g)$ extending 
the given inclusion of $\wh{\t}^\C$. To see this, take $\xi\in \Alc_I$ 
(defining a trivialization $\g_I\cong \g_{I,(\exp\xi)}=\g_I\times \R$). 
Then the desired lift reads, 
\[ \wh{j}_{I,\xi}\colon \wh{\g}_{I,(\exp\xi)}^\C\to \wh{\ca{L}}(\g),\ \ 
\wh{j}_{I,\xi}(\zeta,s)= j_I(\zeta)+(s+B(\xi,\zeta))\mf{c}.\]
By the theory of affine Lie algebras \cite{kac:inf}, 
there exists a unitarizable $\wh{L\g}$-module where the central
element $\mf{c}$ acts as $-1$. Unitarizibility means in particular
that the $\wh{\t}$-action exponentiates to a unitary $\wh{T}$-action,
and hence all $\wh{\g}_I$-actions exponentiate to unitary
$\wh{G}_I$-actions. 

\begin{remark}
If $G=\SU(2)$, there is a much simpler construction. We have unique
trivializations $\wh{G}_0\cong G\times\U(1)\cong \wh{G}_1$. 
The inclusion of $\wh{G}_{01}=\wh{T}$ into $\wh{G}_{1}$ is given 
by $(t,z)\mapsto (t,t^\rho z)$ where $\rho\in\Lambda^*$ generates the
weight lattice. Let $\H=L^2(G)$, and denote by $\H_{(r)}$ the subspace on which $T$
acts with weight $r$. Since all $\H_{(r)}$ are infinite-dimensional,
there exists a unitary transformation $\Psi$ with $\Psi\colon
\H_{(r)}\to \H_{(r+1)}$. Extend the given $G$ action to
$\wh{G}_0=G\times\U(1)$ by letting $\U(1)$ acts with weight $-1$, and
let $\wh{G}_1=G\times\U(1)$ act by the conjugate of this action by the
automorphism $\Psi$. Let $\wh{G}_{01}$ act as a subgroup of $G_0$.
Then these actions satisfy (ii).
\end{remark}

With $\H$ as in the Lemma, put $\A_I=G\times_{G_I}\K(\H)$. For $J\subset I$, the map
$\phi_I^J\colon G/G_I\to G/G_J$ is covered by a homomorphism of
Dixmier-Douady bundles, $\A_I\to \A_J$. 
Hence we may define a $G$-Dixmier-Douady bundle, 
\begin{equation}\label{eq:Asim}
 \A=\coprod_{I} (\A_I\times\Alc_I)/\sim\end{equation}
with identifications similar to those in \eqref{eq:Gsim}.  By
construction, the central extension of $G_I$ defined by the
restriction $\A|_{\exp(\Alc_I)}$ coincides with $\wh{G}_I$.
Hence, the Dixmier-Douady class $\on{DD}_G(G,\A)$ is a generator of
$H^3_G(G)\cong \Z$.

\section{Conjugacy classes}\label{sec:conjclass}
As is well-known, coadjoint orbits $\O\subset \g^*$ carry a
distinguished invariant complex structure, hence a
$\Spin_c$-structure. If $\O$ admits a pre-quantum line bundle $L\to
\O$ (i.e. a line bundle with curvature equal to the symplectic form),
one may twist the original $\Spin_c$-structure by this line bundle.
The resulting equivariant index is the irreducible representation
parametrized by $\O$. In this Section, we will describe a similar
picture for conjugacy classes $\Co\subset G$.

\subsection{Pull-back to conjugacy classes}\label{subsec:GT}
Given $\xi\in\Alc$, define a $G$-equivariant map $\Psi\colon
G/T\to G,\ gT\mapsto \Ad_g(\exp\xi)$. The pull-back $\Psi^*\A$ 
admits a canonical Morita trivialization, defined by the Hilbert space bundle 
$G\times_T \H$. More generally, for any $l\in\Z$ and any weight 
$\mu\in\Lambda^*$ there is a Morita trivialization, 
\begin{equation}\label{eq:mor1}
 \K(\E)\xra{\cong} \Psi^*\A^l,\ \ \ \E=G\times_T(\H^l\otimes\C_\mu)
\end{equation}
where $\C_\mu$ is the 1-dimensional 1-dimensional $T$-representation
of weight $\mu$. 
Dixmier-Douady bundles over $G$, together with Morita trivializations 
of their pull-backs by $\Psi$ are classified by the relative 
cohomology group $H^3_G(\Psi)$. (See Appendix \ref{sec:relative}.)
The map $\Psi=:\Psi_1$ is equivariantly homotopic to the constant map 
$\Psi_0\colon gT\mapsto e$, by the homotopy $\Psi_t(gT)=\exp(t\Ad_g(\xi))$.
Hence $H^3_G(\Psi)=H^3_G(\Psi_0)=H^2_G(G/T)\oplus H^3_G(G)$. 
Identifying $H^2_G(G/T)=H^2_T(\pt)=\Lambda^*$ and $H^3_G(G)=\Z$, 
we obtain an isomorphism
\[ H^3_G(\Psi)=\Lambda^*\oplus \Z,\]
The element $(\mu,l)\in H^3_G(\Psi)$ is realized by the 
Morita trivialization \eqref{eq:mor1}.

Now let $\Co$ be the conjugacy class of $\exp(\xi)$, and $\Phi\colon
\Co\to G$ the inclusion. Let $\pi\colon G/T\to \Co$ be the 
$G$-invariant projection such that $\Psi=\Phi\circ\pi$. 
We obtain a map of long exact sequences in relative cohomology,
\[ \begin{CD} 
\cdots @>>> 0 @>>> H^2_G(\Co) @>>> H^3_G(\Phi) @>>> H^3_G(G) @>>> H^3_G(\Co) @>>>\cdots\\
@. @VVV @VVV @VVV @VV{=}V @VVV @.\\
\cdots @>>> 0 @>>> H^2_T(\pt) @>>> H^3_G(\Psi) @>>> H^3_G(G) @>>> 0 @>>>\cdots\\
\end{CD}\]
Since the map $H^2_G(\Co)\to H^2_G(G/T)$ is injective, the 5-Lemma
implies that the map $H^3_G(\Phi)\to H^3_G(\Psi)$ is
injective.  The exact sequence in the bottom row splits (see Section
\ref{subsec:GT}), and hence we obtain an injective map,
$H^3_G(\Phi)\to H^3_G(\Psi)=\Lambda^*\oplus \Z$. 

By a parallel discussion with real coefficients, there is an isomorphism 
$H^3_G(\Psi,\R)=\t^*\oplus \R$ and an inclusion of $H^3_G(\Phi,\R)$.

\subsection{Pre-quantization of conjugacy classes}
We return to Cartan's de Rham model for $H_G^\bullet(M,\R)$ (cf.  the
proof of Lemma \ref{lem:torusres}) with $\eta_G\in\Om^3_G(G)$
representing the generator of $H^3_G(G)$.  The conjugacy class $\Co$
carries a unique invariant 2-form $\om\in\Om^2(\Co)^G\subset
\Om^2_G(\Co)$ with the property \cite{al:mom,gu:gr},
\begin{equation}\label{eq:qham}
\d_G\om=\Phi^*\eta_G.\end{equation}
The triple $(\Co,\om,\Phi)$ is an example of a quasi-Hamiltonian
$G$-space in the terminology of \cite{al:mom}. Equation
\eqref{eq:qham} together with $\d_G\eta_G=0$ say that
$(\om,\eta_G)\in\Om^3_G(\Phi)$ is a relative equivariant cocycle.  Let
$[(\om,\eta_G)]$ be its class in $H^3_G(\Phi,\R)$.
\begin{lemma}
The inclusion $H^3_G(\Phi,\R) \to \t^*\oplus \R$ takes the class
$[(\om,\eta_G)]$ to the element $(B^\flat(\xi),1)$. 
\end{lemma}
\begin{proof}
Let $h_t\colon \Om_G^\bullet(G)\to \Om_G^{\bullet-1}(G/T)$ be the
homotopy operator defined by the family of maps $\Psi_t$. Thus
$\d\circ h_t+h_t\circ \d=\Psi_t^*-\Psi_0^*$. Then 
\[ \Om_G^\bullet(\Psi_t)\to \Om_G^\bullet(\Psi_0),\ \ 
(\alpha,\beta)\mapsto
(\alpha-h_t(\beta),\beta)\] 
is an isomorphism of chain complexes, inducing the isomorphism
$H_G^\bullet(\Psi_t,\R)\to H_G^\bullet(\Psi_0,\R)$. Hence, the
isomorphism $H^3_G(\Psi_1,\R)\to H_G^\bullet(\Psi_0,\R)$ takes
$[(\om,\eta_G)]$ to $[(\om-h_1^*\eta_G,\eta_G)]$. 

The family of maps $\Psi_t$ is a composition of the map $f\colon G/T\to \g,\ gT\mapsto
\Ad_g(\xi)$ with the family of maps $\g\to G,\ \zeta\mapsto
\exp(t\zeta)$. Let $j_t\colon \Om^\bullet_G(G)\to \Om^{\bullet-1}(\g)$
be the homotopy operator for the second family of maps. Then
$h_t=f^*\circ j_t$. By \cite{me:ge}, we have 
$j_1\eta_G=\varpi_G$, where $\varpi_G\in \Om^2_G(\g)$ is of the form
$\varpi_G(\zeta)|_\xi=\varpi|_\xi-B(\xi,\cdot)$. 
It follows that the image of $[(\om,\eta_G)]$ under the map to
$\t^*\oplus\R$ is $(B^\flat(\xi),1)$.
\end{proof}

As a special case of pre-quantization of group-valued moment maps \cite{al:onq}, we
define:
\begin{definition}
A \emph{level $k\in\Z$ pre-quantization} of a conjugacy class $\Co$ is
a lift of the class $k\,[(\omega,\eta_G)]\in H^3_G(\Phi,\R)$ 
to an integral class.
\end{definition}
By the long exact sequence in relative cohomology, if $\Co$ admits a level $k$ 
pre-quantization, then the latter is unique (since $H^2_G(\Co)$ has no
torsion). 

\begin{proposition}
  The conjugacy class $\Co$ of the element $\exp\xi$ with $\xi\in\Alc$
  admits a pre-quantization at level $k$ if and only if
  $(B^\flat(k\xi),k)\in \Lambda^*\times\Z$.
\end{proposition}
\begin{proof}
According to the Lemma, $k\,[(\omega,\eta_G)]$ maps to
$(B^\flat(k\xi),k)\in \t^*\times\R$. Since all maps in the commutative diagram
\[ \begin{CD} H^3_G(\Phi) @>>> \Lambda^*\oplus \Z\\
@VVV @VVV\\
H^3_G(\Phi,\R) @>>> \t^*\oplus \R\\
\end{CD} \]
are injective, it follows that $k\,[(\omega,\eta_G)]$ is integral 
if and only if $(B^\flat(k\xi),k)\in \Lambda^*\times\Z$.
\end{proof}

Geometrically, a level $k$ pre-quantization is given by a
$G$-equivariant Morita trivialization of $\Phi^*\A^k$. This can
be seen explicitly, as follows. 
%
\begin{lemma}
Let $\xi\in\Alc_I$, and suppose that $B^\flat(k\xi)
  \in\Lambda^*$. Then the $k$-th power of the central extension of
  $G_I$ admits a unique trivialization $G_I\to \wh{G}_I^{(k)}$
  extending the map
\begin{equation}\label{eq:rrr}
 T\to \wh{T}^{(k)}=T\times\U(1),\ h\mapsto
(h,h^{B^\flat(k\xi)}).\end{equation}
\end{lemma}

\begin{proof}
  Uniqueness is clear, since a trivialization $G_I\to \wh{G}_I^{(k)}$
  is uniquely determined by its restriction to $T$.  For existence,
  recall that $\xi\in\Alc_I$ determines an identification
  $\wh{G}_I\cong \wh{G}_{I,(t)}=\wt{G}_I\times_{\pi_1(G_I)}\U(1)$,
  where $t=\exp\xi$, and using the homomorphism $\varrho_{t,I}\colon
  \pi_1(G_I)=\Lambda/\Lambda_I\to \U(1),\ \lambda\mapsto
  t^{B^\flat(\lambda)}$. The powers $\wh{G}^{(l)}$ are obtained
  similarly, using the $l$-th powers of the homomorphism
  $\varrho_{t,I}$. Since $B^\flat(k\xi)$ is a weight, we have
  $(\varrho_{t,I})^k=1$. This defines a trivialization,
\[ \wh{G}_I^{(k)}\cong \wh{G}_{I,(t)}=G_I\times\U(1).\]
According to \eqref{eq:standard}, this isomorphism intertwines the
standard inclusion of $\wh{T}^{(k)}\to \wh{G}_I^{(k)}$ with the map
\[ \wh{T}=T\times\U(1)\to  G_I\times\U(1),\ 
(h,z)\mapsto (h,h^{-B^\flat(k\xi)}z).\] 
The composition of this map with \eqref{eq:rrr} is $h\mapsto (h,1)$,
as required.
\end{proof}

Let $\Phi\colon \Co\hra G$ be the conjugacy class of $t=\exp\xi$, and
let $I$ be the unique index set such that $\xi$ lies in the relative
interior of $\Alc_I$.  If $\Co$ is pre-quantizable at level $k$, so
that $B^\flat(k\xi)\in\Lambda^*$, the Lemma defines a trivialization
of $G_I^{(k)}$. Hence, its action on $\H^k$ descends to an action of
$G_I$, and the Hilbert bundle $\E=G\times_{G_I}\H^k$ defines a Morita
trivialization of $\Phi^*\A^k$.

\begin{proposition}
The relative Dixmier-Douady class $\on{DD}_G(\A,\E,\psi)\in H^3_G(\Phi)$ is an
integral lift of the class $k[(\om,\eta_G)]\in H^3_G(G,\Co,\R)$.
\end{proposition}

\begin{proof}
We have to show that the image of $\on{DD}_G(\A,\E,\psi)$ in 
$H^3_G(\Psi)=\Lambda^*\oplus \Z$ is $(B^\flat(k\xi),k)$. But this 
follows from the discussion in the last Section, since the pull-back of $\E$ to $G/T$ is 
\[ \pi^*\E=G\times_T (\H^k\otimes \C_{B^\flat(k\xi)}).\]
\end{proof}

\subsection{The $\cox$-th power of the Dixmier-Douady bundle}\label{subsec:coxpower}
For any coadjoint orbit $\O\subset \g^*$, the compatible complex
structure defines a $G$-invariant $\Spin_c$-structure, i.e. Morita
trivialization of $\Cl(T\O)$.  We show show that similarly, for all
conjugacy classes $\Co\subset G$, there is a distinguished Morita
isomorphism between $\Cl(T\Co)$ and $\A^{\cox}|_\Co$, where $\cox$ is
the dual Coxeter number. That is, conjugacy classes carry a canonical
`twisted $\Spin_c$-structure'. There are examples of conjugacy classes
that do not admit invariant $\Spin_c$-structures, let alone invariant
complex structures.

We will need some additional notation.  Let
$\mf{S}_0=\{\alpha_1,\ldots,\alpha_l\}$, $l=\on{rank}(G)$, be a set of
simple roots for $\g$, relative to our choice of fundamental Weyl
chamber. We denote by $\alpha_0=-\alpha_{\on{max}}$ minus the highest
root, and let
\[\mf{S}=\mf{S}_0\cup\{\alpha_0\}=\{\alpha_0,\ldots,\alpha_l\}.\] 
Thus  $\Alc\subset\t_+$ is the $l$-simplex cut out by the inequalities
$\l\alpha_i,\cdot\r+\delta_{i,0}\ge 0$ for $i=0,\ldots,l$, 
and $\t_+$ is cut out by the hyperplanes with $i>0$.  The roots of
$G_I$ are those roots $\alpha$ of $G$ for which $\l\alpha,\xi\r\in\Z$,
and a set of simple roots is
\[ \mf{S}_I=\{\alpha_i\in\mf{S}|\ i\not\in I\}\]
That is, the Dynkin diagram of $G_I$ is obtained from the extended
Dynkin diagram of $G$ by removing the vertices labeled by $i\in I$.
Let $\rho$ be the half-sum of positive roots of $G$,
$\rho^\sharp=B^\sharp(\rho)$, and let
\[ \cox=1+\l\alpha_{\on{max}},\rho^\sharp\r\]
be the dual Coxeter number. 
\begin{theorem}
  For any conjugacy class $\Phi\colon \Co\hra G$, there is a
  distinguished $G$-equivariant Morita isomorphism $\Phi^* \A^\cox\simeq \Cl(T\Co)$.
\end{theorem}
\begin{proof}
We have to construct a Morita trivialization $(\E,\psi)$ of 
$\Phi^*\A^\cox\otimes \Cl(T\Co)$ (recall that Clifford algebras
satisfy $\Cl(V)\cong \Cl(V)^{\on{opp}}$). 
  Let $\xi\in\Alc\cong G/\Ad(G)$ be the point of the alcove
  corresponding to $\Co$, and $I$ the index set such that $\xi$ lies
  in the interior of $\Alc_I$. By construction,
  $\Phi^*\A^{\cox}=G\times_{G_I}\K(\H^{\cox})$, while
  $\Cl(T\Co)=G\times_{G_I}\Cl(\g_I^\perp)$ where $\g_I^\perp$ is the
  orthogonal complement of $\g_I$ in $\g$. Hence 
\[ \Phi^*\A^\cox\otimes \Cl(T\Co)=G\times_{G_I}(\K(\H^\cox)\otimes \Cl(\g_I^\perp)),\]
and the Theorem is reduced to the following Lemma.
\end{proof} 
\begin{lemma}\label{lem:cox}
For each $I$ there is a canonical $G_I$-equivariant Morita trivialization
\[\C\simeq \K(\H^\cox)\otimes \Cl(\g_I^\perp).\] 
\end{lemma}
\begin{proof}
  The central extension of $\SO(\g_I^\perp)$ defined by its
  homomorphism to $\on{Aut}(\Cl(\g_I^\perp))$ is by definition the
  group $\Spin_c(\g_I^\perp)$. Hence, the central extension of $G_I$
  defined by its action on $\Cl(\g_I^\perp)$ equals the pull-back of
  $\Spin_c(\g_I^\perp)$ under the homomorphism $G_I\to
  \SO(\g_I^\perp)$. That is, it is of the form
\[ \wt{G}_I\times_{\pi_1(G_I)}\U(1)\]
where $\wt{G}_I$ is the universal covering group, and the homomorphism
$\pi_1(G_I)\to \U(1)$ is defined by the commutative diagram, 
\[\begin{CD}
1 @>>> {\pi_1(G_I)} @>>> {\wt{G}_I} @>>> {G_I} @>>> 1  \\
@. @VVV @VVV @VVV @.\\
1 @>>> \U(1) @>>> {\Spin_c(\g_I^\perp)} @>>> {\on{SO}(\g_I^\perp)} @>>> 1  
\end{CD}
\]
Let $\Lambda_I$ be the co-root lattice of $G_I$, so that
$\pi_1(G_I)=\Lambda/\Lambda_I$.  By a direct calculation (cf.
Sternberg \cite[Section 9.2]{st:lie}), the homomorphism $\pi_1(G_I)\to
\U(1)$ is
\begin{equation}\label{eq:newhom}
 \pi_1(G_I)=\Lambda/\Lambda_I\to \U(1),\ \lambda\mapsto e^{\tpi
 \l\rho-\rho_I,\lambda\r}
\end{equation}
where $\rho$ is the half-sum of positive roots of $G$, and $\rho_I$
is the half-sum of positive roots of $G_I$, relative to the given
system $\mf{S}_I$ of simple roots. Let
\begin{equation}\label{eq:nuI}
\nu_I=\f{1}{\cox}(\rho-\rho_I),\ \ \ \nu_I^\sharp=B^\sharp(\nu_I).
\end{equation}
The element $\nu_I^\sharp$ is contained in the the interior of
the face $\Alc_I$ (see e.g. \cite{me:can}). 
Hence, the homomorphism \eqref{eq:newhom} is
just the $\cox$-th power of the homomorphism $\varrho_{t,I},\ t=\exp\nu_I^\sharp$ in the
definition of $\wh{G}_{I,(t)}\cong \wh{G}_I$. That is, we have a
pull-back 
diagram 
\[\begin{CD} \wh{G}_I^{(\cox)} & @>>> &\Spin_c(\g_I^\perp)\\
@VVV & & @ VVV \\
G_I & @>>> & \on{SO}(\g_I^\perp)
\end{CD}\]
There is an explicit spinor module $\sS_I$ for $\Cl(\g_I^\perp)$, constructed
as follows. Let $\n_+\subset \g^\C$ and $\n_{I,+}\subset\g_I^\C$
be the sum of root spaces for positive roots of $G$ and $G_I$,
respectively. (Here positivity is defined by the respective sets 
$\mathfrak{S}_0,\mathfrak{S}_I$ of simple roots.) Since 
projection to the real part identifies $\n_+\cong \t^\perp$ (as real
vector spaces), $\sS=\wedge\n_+$ is a spinor module for
$\Cl(\t^\perp)$. Similarly $\sS^I=\wedge\n_{I,+}$  is a spinor module for 
$\Cl(\g_I\cap \t^\perp)$. The spinor module $\sS_I$ over $\Cl(\g_I^\perp)$ is
\begin{equation}\label{eq:spinormod}
 \sS_I=\on{Hom}_{\Cl(\g_I\cap \t^\perp)}(\sS^I,\sS).\end{equation}
The Clifford action on $\sS_I$ restricts to a unitary representation
of $\Spin_c(\g_I^\perp)\subset \Cl(\g_I^\perp)$ where the central
circle acts with weight $1$. The $\K(\H^\cox)\otimes\Cl(\g_I)$-module
$\ca{H}^{ \cox}\otimes \sS_I$ with the resulting action of $G_I$ gives
the desired Morita trivialization. 
\end{proof}
The spinor module \eqref{eq:spinormod} is $T$-equivariant. This has
the following consequence: 
\begin{proposition}\label{prop:pull1}
Let $\Co$ be the conjugacy class of $\exp\xi,\ \xi\in \Alc$. The pull-back 
of $\Cl(T\Co)$ under the projection map 
\[ \pi\colon G/T\to \Co,\ gT\mapsto \Ad_g(\exp(\xi)).\]
admits a canonical $G$-equivariant Morita trivialization
\begin{equation}\label{eq:conn} 
\C\simeq \pi^*\Cl(T\Co).\end{equation}
\end{proposition}
\begin{proof}
  Let $I$ be the index set such that $G_I$ is the stabilizer of
  $\exp\xi$. We have $\pi^*\Cl(T\Co)=\Cl(\pi^*T\Co)=G\times_T
  \Cl(\g_I^\perp)$. Hence we need a $T$-equivariant Morita
  trivialization of $\Cl(\g_I^\perp)$, and  this is provided by
  $\sS_I$. 
\end{proof}
If the conjugacy class $\Co$ is pre-quantized at level $k$, the Morita
equivalences $\Cl(T\Co)\simeq \Phi^*\A^\cox$ and $\C\simeq \Phi^*\A^{k}$,
combine to a Morita equivalence
\begin{equation}\label{eq:prequantized}
 \Cl(T\Co)\simeq \Phi^* \A^{k+\cox}\end{equation}
Recall $\Psi=\Phi\circ \pi\colon G/T\to G$.  The composition of the
Morita equivalences \eqref{eq:conn} and $\Cl(T\Co)\simeq
\Phi^*\A^{\cox}$ is the Morita trivialization $\C\simeq
\Psi^*\A^{\cox}$ defined by the bundle $G\times_T\H^\cox$.  It is thus
labeled by $(0,\cox)\in \Lambda^*\oplus \Z$.  Hence, in the
pre-quantized case, the composition of \eqref{eq:conn} and
\eqref{eq:prequantized} is the Morita trivialization of
$\Psi^*\A^{k+\cox}$ parametrized by $(B^\flat(k\xi),k+\cox)\in
\Lambda^*\oplus \Z$.

\subsection{Quantization of conjugacy classes}
The twisted equivariant $K$-homology group
\[ K_\bullet^G(G,\A^{k+\cox})\]
carries a ring structure, with product given by the cross-product for
$G\times G$, followed by push-forward under group multiplication
$\Mult\colon G\times G\to G$. Indeed, since
$\on{Mult}^*x=\pr_1^*x+\pr_2^*x$ for all $x\in H^3_G(G,\Z)$, the
pull-back bundle $\Mult^*\A^{k+\cox}$ over $G\times G$ is (stably)
isomorphic to $\pr_1^*\A^{k+\cox}\otimes\pr_2^*\A^{k+\cox}$. The
choice of a stable isomorphism defines a push-forward map
\[ K_\bullet^G(\on{Mult})\colon 
K_\bullet^G(G,\A^{k+\cox})\otimes K_\bullet^G(G,\A^{k+\cox}) \to
K_\bullet^G(G,\A^{k+\cox}),
\]
where we identify $ K_\bullet^G(G,\A^{k+\cox})\otimes
K_\bullet^G(G,\A^{k+\cox})\cong K_\bullet^G(G\times
G,\pr_1^*\A^{k+\cox}\otimes\pr_2^*\A^{k+\cox})$.  Since $H^2_G(G\times
G)=0$ this map in $K$-homology is independent of the choice of stable
isomorphism, and by a similar reasoning the resulting product is
commutative and associative. (For non-simply connected groups
$G$, the existence of ring structures on the twisted K-homology is a much more subtle matter \cite{tu:rin}.)

The inclusion $\iota\colon \{e\}\hra G$
of the group unit induces a ring homomorphism
\begin{equation}\label{eq:quotientmap}
 K_\bullet^G(\iota)\colon R(G)=K_\bullet^G(\pt)\to 
K_\bullet^G(G,\A^{k+\cox}).\end{equation}
\begin{theorem}[Freed-Hopkins-Teleman] \label{th:fht}
For all non-negative integers $k\ge 0$ the ring homomorphism
  \eqref{eq:quotientmap} is onto, with kernel the level $k$ fusion
  ideal $I_k(G)\subset R(G)$. That is, $K_1^G(G,\A^{k+\cox})=0$, while
  $K_0^G(G,\A^{k+\cox})$ is canonically isomorphic to the level $k$
  fusion ring, $R_k(G)=R(G)/I_k(G)$. 
\end{theorem}
We will explain a proof of this Theorem in Section \ref{sec:spec}
below.
\begin{remark}
  It is also very interesting to consider the non-equivariant twisted
  $K$-homology rings $K_\bullet(G,\A^{k+\cox})$.  These are studied
  are in the work of V. Braun \cite{br:tw} and C. Douglas
  \cite{dou:twi}.
\end{remark}
The ring $R_k(G)$ may be defined as the ring of level $k$
projective representations of the loop group $LG$, or in
finite-dimensional terms (cf. \cite{al:fi}): Let  
\[ \Lambda^*_k=\Lambda^*\cap B^\flat(k\Alc)\] 
be the set of \emph{level $k$ weights}. 
Identify $R(G)$ with ring of characters of $G$. Then
$R_k(G)=R(G)/I_k(G)$, where $I_k(G)$ is the vanishing ideal of the set
of elements $\{t_\nu\in T,\ \nu\in\Lambda^*_k\}$ where
\[ 
t_\nu=\f{1}{k+\cox}B^\sharp(\nu+\rho).\]
It turns out that as an additive group, $R_k(G)$ is freely generated
by the images of irreducible characters $\chi_\mu$ for $\mu\in
\Lambda^*_k$.  Thus $R_k(G)=\Z[\Lambda^*_k]$ additively.
\begin{remark}
If $G$ has type $ADE$ (so that all roots have equal length), the lattice
$B^\sharp(\Lambda^*)\subset\t$ is identified with the set of elements
$\xi\in \t$ with $\exp\xi\in Z(G)$, the center of $G$.  Hence the
ideal $I_k(G)$ may be characterized, in this case, as the vanishing
ideal of the set of all $g\in G^{\on{reg}}$ such that $g^{k+\cox}\in
Z(G)$.
\end{remark}
\begin{remark}
  Freed-Hopkins-Teleman compute twisted $K$-homology groups of $G$ for
  arbitrary compact groups, not necessarily simply connected. The case
  of simple, simply connected groups considered here is considerably
  easier than the general case.
\end{remark}
Suppose $\Phi\colon \Co\hra G$ is the conjugacy class of $\exp\xi,\
\xi\in\Alc$, pre-quantized at level $k\ge 0$. Thus
$\mu:=B^\flat(k\xi)$ is a weight. The Morita equivalence
\eqref{eq:prequantized} defines a push-forward map in K-homology,
\begin{equation}\label{eq:pushfor}
 K_0^G(\Phi)\colon K_0^G(\Co,\Cl(T\Co))\to 
K_0^G(G,\A^{k+\cox})\end{equation}
where $\Phi\colon \Co\hra G$ is the inclusion. 

\begin{theorem}
The push-forward map \eqref{eq:pushfor} takes the fundamental class
$[\Co]\in K_0^G(\Co,\Cl(T\Co))$ to the equivalence class of the
character $\chi_\mu$ in $R_k(G)=R(G)/I_k(G)$.  
\end{theorem}
\begin{proof}
  Let $\pi\colon G/T\to \Co$ and $\Psi=\Phi\circ \pi \colon G/T\to G$
  be as in Section \ref{subsec:GT}, and recall that $\Psi=\Psi_1$ is
  equivariantly homotopic to the constant map $\Psi_0$ onto $e\in G$.
  That is, the diagram
\[ \begin{CD}
G/T @>>{\pi}> \Co\\
@VVV @VV{\Phi}V\\
\pt @>>{\iota_e}> G,
\end{CD}\]
commutes up to a $G$-equivariant homotopy. As discussed at the end of Section
\ref{subsec:coxpower}, the Morita trivialization $\C\simeq
\pi^*\Cl(T\Co)$ from \eqref{eq:conn}  combines with
\eqref{eq:prequantized} to a Morita trivialization of
\[\Psi^*\A^{k+\cox}\cong \K(G\times_T(\C_{\mu}\otimes\H^{k+\cox})).\] 
It is related to the obvious trivialization $\Psi_0^*\A^{k+\cox}\cong
\K(G\times_T \H^{k+\cox})$ by a twist by the line bundle $G\times_T
\C_\mu$. We obtain a commutative diagram 
\[ \begin{CD}
K_0^G(G/T)=R(T) @>>> K_0^G(\Co,\Cl(T\Co))\\
@VVV @VVV\\
K_0^G(\pt)=R(G) @>>> K_0^G(G,\A^{k+\cox})=R_k(G),
\end{CD}\]
where the lower horizontal map is defined by the `obvious' 
trivialization $\iota_e^*\A^{k+\cox}\cong \K(\H^{k+\cox})$, and the 
left vertical map is twist by the line bundle $G\times_T\C_\mu$
(acting as an automorphism of $K_0^G(G/T)$, followed by the push-forward.

Using the Morita trivialization $\Cl(T(G/T))\cong \K(G\times_T\sS)$
(where $\sS=\wedge\n_+$) to identify
$K_0^G(G/T)=K_0^G(G/T,\Cl(T(G/T)))$, the left vertical map takes the
fundamental class $[G/T]$ to the irreducible character $\chi_\mu \in
R(G)=K_0^G(\pt)$. On the other hand, the upper horizontal map takes
$[G/T]$ to $[\Co]$, which then goes to $\Phi_*[\Co]$. We conclude that
$\Phi_*[\Co]\in R_k(G)$ is the image of $\chi_\mu$ under the quotient
map $R(G)\to R_k(G)$.
\end{proof}

\subsection{Twisted $K$-homology of the conjugacy classes}
Suppose $\Phi\colon \Co\hra G$ is an arbitrary conjugacy class (not
necessarily pre-quantized) corresponding to $\xi\in\Alc$. Let $I$ be
the index set such that $\xi$ is in the interior of $\Alc_I$, thus
$\Co=G/G_I$. The twisted $K$-homology group
$K_q^G(\Co,\Phi^*\A^{k+\cox})= K_q^G(G/G_I,\ \A_I^{k+\cox})$ is
computed as in Section \ref{subsec:twi}, using Lemma \ref{lem:cox}
\[ \begin{split}
K_q^G(G/G_I,\ \A_I^{k+\cox})
&=K_q^{G_I}(\pt,\ \K(\H)^{k+\cox}\otimes\Cl(\g_I^\perp))
\\&=K_q^{G_I}(\pt,\ \K(\H)^k).
\end{split}\]
The central extension of $G_I$ defined by its action on $\K(\H)^k$ is
$\wh{G}_I^{(-k)}$. Hence, the last group is isomorphic to the
Grothendieck group of $\wh{G}_I^{(-k)}$-representations where the
central circle acts with weight $-1$, or equivalently the Grothendieck
group $R(\wh{G}_I)_k$ of $\wh{G}_I$-representations where the central
circle acts with weight $k$. That is, $K_1^G(G/G_I,\
\A_I^{k+\cox})=0$, while 
\begin{equation}\label{eq:kconj}
 K_0^G(G/G_I,\ \A_I^{k+\cox})\cong R(\wh{G}_I)_k
\end{equation}
as $R(G)$-modules. (The module structure is given by the restriction
homomorphism $R(G)\to R(G_I)=R(\wh{G}_I)_0$, which acts on
$R(\wh{G}_I)$ by multiplication.) If $J\subset I$, we have a natural
map $\phi_I^J\colon G/G_I\to G/G_J$ covered by a map of Dixmier-Douady
bundles $\A_I\to \A_J$. Hence we obtain a push-forward map,
\begin{equation}\label{eq:push}
K_0^G(\phi_{I}^J)\colon K_0^G(G/G_I,\A_I^{k+\cox})\to
K_0^G(G/G_J,\A_J^{k+\cox})\end{equation}
On the other hand, the collections of simple roots
$\mf{S}_J\subset\mf{S}_I$ determine complex structures on the
homogeneous spaces $G_J/G_I=\wh{G}_J/\wh{G}_I$. This defines
holomorphic induction maps $\on{ind}_I^J\colon R(\wh{G}_I)\to
R(\wh{G}_J)$, restricting to $R(G)$-module homomorphisms
\begin{equation}\label{eq:holind} \on{ind}_I^J\colon R(\wh{G}_I)_k\to R(\wh{G}_J)_k.\end{equation}
From the discussion in Section \ref{subsec:twi}, we see:
\begin{proposition}\label{prop:induction}
The identifications $K_0^G(G/G_I,\A_I^{k+\cox})\cong R(\wh{G}_I)_k$
intertwine the push-forward maps \eqref{eq:push}
with the holomorphic induction maps \eqref{eq:holind}. 
\end{proposition}

\section{Computation of   $K_\bullet^G(G,\A^{k+\cox})$}\label{sec:spec}
The Dixmier-Douady bundle $\A\to G$, as described in \eqref{eq:Asim},
may be viewed as the geometric realization of a co-simplicial
Dixmier-Douady bundle, with non-degenerate $p$-simplices the bundle
$\coprod_{|I|=p+1}\A_I$ over $\coprod_{|I|=p+1}G/G_I$. This defines a
spectral sequence computing the $K$-homology group
$K_\bullet^G(G,\A^{k+\cox})$, in terms of the known $K$-homology
groups $K_\bullet^G(G/G_I,\A_I^{k+\cox})=R(\wh{G}_I)_k$ and the
holomorphic induction maps between these groups.  As it turns out, the
spectral sequence collapses at the $E_2$-stage, and computes the level
$k$ fusion ring.

\subsection{The spectral sequence for $K_\bullet^G(G,\A^{k+\cox})$}
The construction \eqref{eq:Asim} of $\A\to G$ as a quotient of $\coprod_I
\A^I\times\Alc_I \to \coprod_I G/G_I\times\Alc_I$ may be thought of as
the geometric realization of a `co-simplicial Dixmier-Douady bundle'.
See \cite{se:cl} and \cite{mo:no} for background on co-simplicial
(semi-simplicial) techniques. Here the $G$-Dixmier-Douady bundles
\[ \coprod_{|I|=p+1}\A_I\to \coprod_{|I|=p+1} G/G_I\] 
are the \emph{non-degenerate} $p$-simplices; the full set of $p$-simplices is a union 
$\coprod_{f}\A_{f([p])}\to
  \coprod_{f} G/G_{f([p])}$
over all non-decreasing maps $f\colon
  [p]=\{0,\ldots,p\}\to \{0,\ldots,l\}$.
By the theory of co-simplicial spaces (see \cite[Section 5]{se:cl}),
one obtains a spectral sequence $E^1_{p,q}\Rightarrow
K_{p+q}^G(G,\A^{k+\cox})$ where
\begin{equation}\label{eq:spectral}
  E^1_{p,q}=\bigoplus_{|I|=p+1} K_q^G(G/G_I,\A_I^{k+\cox}).
\end{equation}
The differential $\d^1\colon E^1_{p,q}\to E^1_{p-1,q}$ is given on
$K_q^G(G/G_I,\A_I^{k+\cox})$ as an alternating sum,
\[ \d^1=\sum_{r=0}^p\,(-1)^r K_q^G(\phi_I^{\delta_rI}).\]
Here $\delta_r I$ is obtained from $I$ by omitting the $r$-th entry:
$\delta_r I=\{i_0,\ldots,\wh{i}_r,\ldots,i_p\}$ for
$I=\{i_0,\ldots,i_p\}$ with $i_0<\cdots <i_p$. Recall that 
$\phi_I^J\colon G/G_I\to G/G_J$ are the natural maps for $J\subset I$.

By mod 2 periodicity of the $K$-homology, we have
$E^1_{p,q}=E^1_{p,q+2}$. Since the groups $G_I$ are connected, and
since $\dim G/G_I$ is even, one has $K_1^G(G/G_I,\A_I^{k+\cox})=0$,
thus $E^1_{\bullet,1}=0$. Hence, the $E^1$-term is described by a
single chain complex $(C_\bullet,\partial)$, where
\[ C_p=E^1_{p,0},\ \ \  \partial=\d^1.\]
The map $R(G)\to K^G_\bullet(G,\A^{k+\cox})$ defined by the inclusion
$\iota\colon e\hra G$ may be also be described by the spectral
sequence.  Think of $\iota$ as the geometric realization of a map of
co-simplicial manifolds, given as the inclusion of
$\{e\}=G/G_{\{0\}}$ into $\coprod_{i=0}^l G/G_{\{i\}}$. The
co-simplicial map gives rises to a morphism of spectral sequences,
$\ti{E}^\bullet\to E^\bullet$, where
\[ \ti{E}^1_{p,q}=\begin{cases} K_q^G(\pt,\C)&\text{if $p=0$}\\
0&\text{otherwise}
\end{cases}
\] 
At the $E^1$-stage, this boils down to a chain map
\begin{equation}\label{eq:unit} R(G)\to C_\bullet\end{equation}
where $R(G)=\ti{E}^1_{0,0}$ carries the zero differential.  
Our goal is to show that the homology of $(C_\bullet,\partial)$ vanishes in positive degrees, while the induced map in homology $R(G)\to H_0(C,\partial)$ is onto, with kernel $I_k(G)$. 

\subsection{The induction maps in terms of weights}
To get started, we express the chain complex in terms of weights of
representations. Recall that $R(T)$ is isomorphic to the group ring
$\Z[\Lambda^*]$. The restriction map $R(G)\to R(T)$ is injective, and
identifies
\[ R(G)\cong \Z[\Lambda^*]^W.\]
%
%
%
%
%
%
Let us next describe $R(\wh{G}_I)_k$ in terms of weights.  Each
$\wh{G}_I$ has maximal torus $\wh{T}=T\times\U(1)$, hence the weight
lattice is
\[ \wh{\Lambda}^*=\Lambda^*\times\Z\subset \wh{\t}^*=\t^*\times\R.\] 
The simple roots for $\wh{G}_I$ are 
$(\alpha_i,0)$ with $\alpha_i\in \mf{S}_I$, the 
corresponding co-roots are 
\begin{equation}\label{eq:coroots}
(\alpha^\vee_i,\delta_{i,0})\in\wh{\t}=\t\times\R,\ \ \
 \alpha_i\in\mf{S}_I.
\end{equation}
These define a fundamental Weyl chamber 
\begin{equation}\label{eq:weylchamber}
 \wh{\t}^*_{I,+}=\{(\nu,s)|\ \l\nu,\alpha^\vee_i\r+s\delta_{i,0} \ge
0,\ \alpha_i\in\mf{S}_I\}\end{equation}
The elements $\nu_I$ satisfy
$\l\nu_I,\alpha^\vee_i\r+\delta_{i,0}=0$.  Hence, $(\nu,s)\in
\wh{\t}^*_{I,+}$ is and only if $\nu-s\nu_I\in \t^*_{I,+}$.
Let $\Lambda^*_{I,k}\subset \Lambda^*$
be the intersection of \eqref{eq:weylchamber} 
with $\Lambda^*\times\{k\}\cong \Lambda^*$. Thus
\[ \Lambda^*_{I,k}=\{\nu\in\Lambda^*|\  \l\nu,\alpha^\vee_i\r+k\delta_{i,0}\ge 0,\ i\not\in I\}\]
labels the irreducible $\wh{G}_I$-representations for which the
central circle acts with weight $k$. The Weyl group $W_I$ of $G_I$ is
also the Weyl group of $\wh{G}_I$. Its action on $\wh{\Lambda^*}$ 
preserves the levels $\Lambda^*\times\{k\}$, 
hence it takes the form $w.(\nu,k)=(w\bullet_k \nu,k)$ for a
\emph{level $k$-action} $\nu\mapsto w\bullet_k \nu$ on $\Lambda^*$. 
Explicitly, 
\begin{equation}\label{eq:shift}
 w\bullet_k \nu=w(\nu-k\nu_I)+k\nu_I.\end{equation}
Fix $k$, and denote by $\Z[\Lambda^*]^{W_I-\on{as}}$ the
anti-invariant part for the $W_I$-action $\nu\mapsto w\bullet_{k+\cox}\nu$ at the shifted level
$k+\cox$. Observe that this space is invariant under the action of $\Z[\Lambda^*]^W$. 
Let
\[\on{Sk}^I\colon \Z[\Lambda^*]\to \Z[\Lambda^*]^{W_I-\on{as}},\ \nu\mapsto 
\sum_{w\in W_I}(-1)^{\length(w)}w\bullet_{k+\cox}\nu \]
denote skew-symmetrization relative to the action at level $k+\cox$. 
For $\mu\in \Lambda^*_k$, let $\chi^I_\mu\in R(\wh{G}_I)_k$ be the
character of the irreducible $\wh{G}_I$-representation of weight
$(\mu,k)$.  
\begin{lemma}
The map $\chi^I_{\mu}\mapsto \on{Sk}^I(\mu+\rho)$
extends to an isomorphism 
\begin{equation}\label{eq:ide}
 R(\wh{G}_I)_k\to  \Z[\Lambda^*]^{W_I-\on{as}}.\end{equation}
Under this isomorphism, the $R(G)\cong \Z[\Lambda^*]^W$-module
structure is given by multiplication in the group ring. Furthermore,
the identification \eqref{eq:ide} intertwines the holomorphic
induction maps $\on{ind}_I^J\colon R(\wh{G}_I)_k\to R(\wh{G}_J)_k$ for
$J\subset I$ with skew-symmetrizations
\[\on{Sk}_I^J=\f{1}{|W_I|}\on{Sk}_J\colon \Z[\Lambda^*]^{W_I-\on{as}}\to 
\Z[\Lambda^*]^{W_J-\on{as}}.\]
\end{lemma}
Note that the statement involves a shift by $\rho$, rather than
$\rho_I$. Thus, even in the case $I=\{0,\ldots,l\}$ where $G_I=T$ and
$W_I=\{1\},\ \rho_I=0$, the identification $R(\wh{T})_k\to
\Z[\Lambda^*]$ involves a $\rho$-shift.
\begin{proof}
  Let $\Lambda^{*,\on{reg}}_{I,k+\cox}$ be the intersection of
  $\Lambda^*\times\{k+\cox\}$ with $\on{int}(\wh{\t}^*_{I,+})$. Since
  obviously $R(\wh{G}_I)_k=\Z[\Lambda^*_{I,k}]$, the first part
 of the Lemma amounts to the assertion that 
\[ \mu\in\Lambda^*_{I,k} \Leftrightarrow \mu+\rho \in
\Lambda^{*,\on{reg}}_{I,k+\cox}.\]
We have $\mu\in \Lambda^*_{I,k}$ if and only if 
$\l\mu,\alpha^\vee_i\r+k\delta_{i,0}\ge 0$ for $i\not\in I$.  Since
$\l\rho,\alpha^\vee_i\r+\cox\delta_{i,0}=1$ this is equivalent to
$\l\mu+\rho,\alpha^\vee_i\r+(k+\cox)\delta_{i,0}\ge 1,\ i\not\in I$,
i.e. $\mu+\rho\in {\Lambda}^{*,\on{reg}}_{I,k+\cox}$ as claimed.  The
assertion about the $R(G)$-module structure is obvious. Finally, for
$J\subset I$ the holomorphic induction map $\on{ind}_I^J$ is given by
\[ \on{ind}_I^J(\chi_{\mu}^I)= (-1)^{\length(w)}
\chi^J_{w\bullet_k(\mu+\rho_J)-\rho_J}\]
if there exists $w\in W_J$ with $w\bullet_k(\mu+\rho_J)-\rho_J\in
{\Lambda}^*_{J,k}$, while $\on{ind}_I^J(\chi_{\mu}^I)=0$ if there
is no such $w$.  
Using \eqref{eq:shift} together with
$\rho_I-k\nu_I=\rho-(k+\cox)\nu_I$ (by the definition of $\nu_I$), this may be
re-written in terms of the action at level $k+\cox$:
\[ w\bullet_k(\mu+\rho_J)-\rho_J=w\bullet_{k+\cox}(\mu+\rho)-\rho.\]
\end{proof}
By combining this discussion with Proposition \ref{prop:induction}, we
have established a commutative diagram
\begin{equation}\label{eq:commd}
 \begin{CD}
K_0^G(G/G_J,\A_J^{k+\cox}) @>>{\cong}> R(\wh{G}_J)_k @>>{\cong}>  \Z[\Lambda^*]^{W_J-\on{as}}\\
@AA{K_0(\phi_I^J)}A @AA{\on{ind}_I^J}A @AA{\on{Sk}_I^J}A\\
K_0^G(G/G_I,\A_I^{k+\cox}) @>>{\cong}> 
R(\wh{G}_I)_k @>>{\cong}>  \Z[\Lambda^*]^{W_I-\on{as}}\\
\end{CD}\end{equation}
We can thus re-express the
chain complex $(C_\bullet,\partial)$ in terms of weights:
\begin{equation}\label{eq:chain1}
 C_p=\bigoplus_{|I|=p+1} \Z[\Lambda^*]^{W_I-\on{as}},\ \ 
\partial \phi^I=\sum_{r=0}^p (-1)^r\on{Sk}_I^{\delta^r I}(\phi^I),\end{equation}
for $\phi^I\in \Z[\Lambda^*]^{W_I-\on{as}}$. The map $R(G)\to
C_0\subset C_\bullet$ given by \eqref{eq:unit} is expressed as the
inclusion of $\Z[\Lambda^*]^{W-\on{as}}$, as the summand corresponding
to $I=\{0\}$.  By construction, $C_\bullet$ is a complex of
$R(G)$-modules, and the map \eqref{eq:unit} is an $R(G)$-module
homomorphism.

\subsection{Fusion ring}
Let us also describe the fusion ring in terms 
of weights. The subset $B^\flat(k\Alc)\subset\t^*$ defining 
the set $\Lambda^*_k=\Lambda^*\cap B^\flat(k\Alc)$ of level $k$
weights is cut out by the
inequalities
\[ \l\nu,\alpha^\vee_i\r+k\delta_{i,0}\ge 0.\]
It is a fundamental domain for the level $k$ action $\nu\mapsto
w\bullet_k \nu$ of the affine Weyl group, generated by the simple
affine reflections
\[ \nu\mapsto \nu-(\l\nu,\alpha^\vee_i\r+s\delta_{i,0}) \alpha_i,\ \
i=0,\ldots,l.\]
This is consistent with our earlier notation: The level $k$ action of
$\Waff$ restricts to the level $k$ action of the subgroup $W_I$,
generated by the affine reflections with $i\not\in I$.  

Let $\Z[[\Lambda^*]]$ be the
$\Z[\Lambda^*]$-module consisting of all functions $\Lambda^*\to \Z$,
not necessarily of finite support. Let
\[ \on{Sk}_{\on{aff}}\colon \Z[\Lambda^*]\to
\Z[[\Lambda^*]]^{\Waff-\on{as}},\ \nu\mapsto \sum_{w\in\Waff} (-1)^{\length(w)}\,w\bullet_{k+\cox}\nu\]
be skew-symmetrization, using the action at the shifted level
$k+\cox$. The map $\mu\mapsto \on{Sk}_{\on{aff}}(\mu+\rho)$ extends to
an isomorphism, $\Z[\Lambda^*_k]\to \Z[[\Lambda^*]]^{\Waff-\on{as}}$.
This identifies
\begin{equation}\label{eq:rkg}
R_k(G)\cong \Z[[\Lambda^*]]^{\Waff-\on{as}}
\end{equation} 
as an abelian group. For any $I$ we have $R(G)=\Z[\Lambda^*]^W$-module homomorphisms $R(\wh{G}_I)_k\to
R_k(G)$,
\begin{equation}\label{eq:mor}
 \Z[\Lambda^*]^{W_I-\on{as}}\to \Z[[\Lambda^*]]^{\Waff-\on{as}},\ \ 
\ \phi_I\mapsto \f{1}{|W_I|}\on{Sk}_{\on{aff}}\phi_I.
\end{equation}
For $I=\{0\}$ we may use the obvious trivialization
$\wh{G}=G\times\U(1)$ to identify $R(G)=R(\wh{G}_0)_k$. The following is clear from the
description of the quotient map $R(G)\to R_k(G)$ (see e.g. \cite{al:fi}):
\begin{lemma}
  The identifications $R(G)=\Z[\Lambda^*]^{W-\on{as}}$ and
  \eqref{eq:rkg} intertwine the quotient map $R(G)\to R_k(G)$ with the
  skew-symmetrization map,
\begin{equation}\label{eq:affind}
\f{1}{|W|}\on{Sk}_{\on{aff}}\colon \Z[\Lambda]^{W-\on{as}}\to \Z[[\Lambda^*]]^{\Waff-\on{as}}.
\end{equation}
In particular, \eqref{eq:rkg} is an isomorphism of 
$R(G)\cong \Z[\Lambda^*]^W$-modules.
\end{lemma}
In fact, we could \emph{define} the ideal $I_k(G)\subset R(G)$ 
as the kernel of the map \eqref{eq:affind}. 
Let $\eps\colon C_0\to R_k(G)$ be the direct sum 
of the morphisms \eqref{eq:mor} for $|I|=1$.

\subsection{A resolution of the $R(G)$-module $R_k(G)$}
\begin{theorem}\label{th:resolution}
For all $k\ge 0$ the chain complex $(C_\bullet,\partial)$ 
defines a resolution 
\[ 0\xra{} C_l\xra{\partial} \cdots \xra{\partial} C_0\xra{\epsilon} 
R_k(G)\xra{} 0\]
of $R_k(G)$ as an $R(G)$-module. 
\end{theorem}

The proof will be given below. As mentioned
in the introduction, Theorem \ref{th:resolution} is implicit in the
work of Kitchloo-Morava \cite{kit:th}.

\begin{remark}
  If $G$ is of type $A_n$ or $C_n$, the central extensions $\wh{G}_I$
  are all trivial, thus $R(\wh{G}_I)_k\cong R(G_I)$ as an
  $R(G)$-module. By the Pittie-Steinberg theorem \cite{st:pi}, if $K$ is a maximal
  rank subgroup of a compact simply connected Lie group $G$, then
  $R(K)$ is free as a module over $R(G)$. Thus $C_\bullet$ is a free
  resolution of the $R(G)$-module $R_k(G)$ in those cases.
\end{remark}

\begin{remark}
Theorem \ref{th:resolution} implies the
Freed-Hopkins-Teleman theorem \eqref{eq:fht}: By acyclicity of the
chain complex $C_\bullet$ the spectral sequence $E^r$ collapses at the
$E^2$-term, with
\[ E^2_{p,q}=E^\infty_{p,q}=\begin{cases} R_k(G)&\text{ if $p=0$ and $q$ even}\\
0&\text{ otherwise }
\end{cases}\]
Since $R_k(G)$ is free Abelian as a $\Z$-module, there are no
extension problems and we 
conclude $K_1^G(G,\A^{k+\cox})=0$, while
\begin{equation}\label{eq:frht}
 K_0^G(G,\A^{k+\cox})=R_k(G)
\end{equation}
as modules over $R(G)$. This isomorphism takes the ring homomorphism 
$R(G)\to K_0^G(G,\A^{k+\cox})$ to the quotient map $R(G)\to R_k(G)$,
hence \eqref{eq:frht} is an isomorphism of rings. 
\end{remark}

The statement of Theorem \ref{th:resolution} can be simplified. 
Indeed, the chain complex $C_\bullet$ breaks up as a direct sum of
sub-complexes $C_\bullet(\mu),\ \mu\in\Lambda^*_k$, given as 
\[ C_p(\mu)=\bigoplus_{|I|=p+1} \Z[\Waff\bullet_{k+\cox}\mu]^{W_I-\on{as}}.\]
Similarly the map $\eps\colon C_0\to R_k(G)$ splits into a direct sum of maps 
\[ \eps\colon C_0(\mu)\to \Z[\Waff\bullet_{k+\cox}\mu]^{\Waff-\on{as}}=
\begin{cases} \Z & \text{ for
$\mu\in\Lambda^{*,\on{reg}}_{k+\cox}$}\\
0 & \text{ otherwise.}\end{cases}\]
Finally the chain map $R(G)\hra C_\bullet$ splits into inclusions of
$\Z[\Waff\bullet_{k+\cox}\mu]^{W-\on{as}}$ as the term corresponding
to $I=\{0\}$. Clearly, $(C_\bullet(\mu),\partial)$ depends only on the
open face $B^\flat((k+\cox)\Alc_J)$ of $B^\flat((k+\cox)\Alc)$
containing $\mu$. Indeed, since
$\Z[\Waff\bullet_{k+\cox}\mu]=\Z[\Waff/W_{J}]$ we have
\[ C_p(J)=\bigoplus_{|I|=p+1} \Z[\Waff/W_{J}]^{W_I-\on{as}}.\]
The differential $\partial$ is again given by anti-symmetrization as
in \eqref{eq:chain1}, but with $\phi^I$ now an element of
$\Z[\Waff/W_{J}]^{W_I-\on{as}}$.  The map $\eps\colon C_0\to R_k(G)$
translates into the zero map $C_0(J)\to 0$ unless $J=\{0\,\ldots,l\}$,
in which case it becomes a map $\eps\colon C_0(J)\to \Z$, given as
the direct sum for $i=0,\ldots,l$ of the maps, 
\[ \Z[\Waff]^{W_i-\on{as}}\to \Z,\ \ \ \sum_w n_w w\mapsto \sum_W
n_w (-1)^{\length(w)}.\]
The map $R(G)\to C_\bullet$ is again the inclusion of the summand of
$C_0(J)$ corresponding to $I=\{0\}$.


%
Theorem \ref{th:resolution} is now reduced to the following simpler statement: 
\begin{theorem}
\label{th:combinatorial}
The homology $H_\bullet(J)$ of the chain complex $C_\bullet(J)$
vanishes in degree $p>0$, while 
\[ H_0(J)=
\begin{cases}0&\mbox{ if }J\not=\{0,\ldots,l\}\\
\Z&\mbox{ if }J=\{0,\ldots,l\}
\end{cases}
\]
In the second case, the isomorphism is induced by the augmentation map
$\eps\colon C_0(J)\to \Z$.
\end{theorem}
 
\subsection{Proof of Theorem \ref{th:combinatorial}}\label{sec:proof}
Throughout this Section, we consider a given face $\Alc_J$ of the
alcove.  We may think of $\Waff/W_J$ as the $\Waff$-orbit of a point
in the interior of the face $\Alc_J$, under the standard action of
$\Waff$ on $\t$.  To be concrete, let us take the point
$\nu_J^\sharp$.
%
%
Denote its orbit by
\[ V=\Waff.\nu_J^\sharp\subset\t.\]
We introduce a length function $\length\colon V\to \Z$, defined in
terms of the function on $\Waff$ as
\[ \length(x)=\on{min}\{\length(w)| w\in\Waff,\ x=w.\nu_J^\sharp\},\ \ x\in V. \]
Geometrically, $\length(x)$ is the number of affine root hyperplanes
in the Stiefel diagram, crossed by a line segment from the a point in
the interior of $\Alc$ to the point $x$.

For any $I$ let $\t_{I,+}$ be defined by the inequalities
 $\l\alpha_i,\cdot\r+\delta_{i,0}\ge 0$ for
 $\alpha_i\in\mf{S}_I$. (Equivalently, it is the affine cone over
 $\Alc$ at $\nu_I^\sharp$.) Then $\t_{I,+}$ is a fundamental domain
 for the $W_I$-action.
Let $V^I\subset \ol{V}^I\subset V$ be the subsets, 
\[ V^I=V\cap \on{int}(\t_{I,+}),\ \ \ \ol{V}^I=V\cap \t_{I,+}.\]
Every $W_I\subset \Waff$-orbit contains a unique point in $\ol{V}^I$.
Thus, if $x\in V$, we may choose  $u\in W_I$ with $u.x\in \ol{V}^I$.
Then
\[ \length(u.x)\le \length(x),\]
with equality if and only if $x\in \ol{V}^I$ and hence $u.x=x$.  

The
elements
\begin{equation}\label{eq:basis}
 \beta_I(x)=\Sk^I(x),\ \ x\in V^I
\end{equation} 
form a basis of the $\Z$-module $\Z[V]^{W_I-\on{as}}$. (Note that if
$x\in \ol{V}^I\backslash V^I$ then $\on{Sk}^I(x)=0$.)  Let us describe the differential in terms of this basis. For $|I|=p+1$ and $x\in V^I$ we have, 
\[\partial\beta_I(x)= \sum_{r=0}^p (-1)^r \on{Sk}^{\delta_r I}(x).\]
In general, the terms $\on{Sk}^{\delta_r I}(x)$ are not standard basis
elements, since $x$ need not lie in $V^{\delta_r I}$. 
Letting $u_r\in W_{\delta^rI}$ be the unique element such that $u_r x\in
V^{\delta^rI}$, we have
\begin{equation}\label{eq:ddd}
 \partial\beta_I(x)=\sum_{r=0}^n
 (-1)^{r+\length(u_r)}\beta_{\delta^rI}(u_r x).\end{equation}

\subsubsection{Computation of $H_0(J)$}
Consider $C_0(J)=\bigoplus_{i=0}^p \Z[V]^{W_i-\on{as}}$. For all $i,j$
and all $x$, the elements $\on{Sk}^j(x),\on{Sk}^i(x)$ are homologous
since they differ by the boundary of $\on{Sk}^{ij}(x)\in C_1(J)$.
Together with $\on{Sk}^j(x)=(-1)^{\length(w)} \on{Sk}^j(wx)$ for $w\in
W_j$, this implies
\[ \on{Sk}^i(x)\sim (-1)^{\length(w)} \on{Sk}^i(wx)\]
for $w\in W_j$. Since the subgroups $W_j$ generate $\Waff$, this
holds in fact for all $w\in\Waff$. Thus 
\[ \on{Sk}^j(w.\nu_J^\sharp)\sim \on{Sk}^i(w.\nu_J^\sharp)\sim (-1)^{\length(w)}\on{Sk}^i(\nu_J^\sharp)\]
for all $i,j$, and all $w\in\Waff$. If $J\not=\{0,\ldots,l\}$, the
choice of any $i\not\in J$ gives $\on{Sk}^i(\nu^\sharp_J)=0$. This proves
$H_0(J)=0$.  Suppose now $J=\{0,\ldots,l\}$.  The augmentation map
$C_0(J)\to \Z$ is described in terms of the basis by
$\beta_i(x)\mapsto (-1)^{\length(x)}$. It has a right inverse $\Z\to
C_0(J),\ 1\mapsto \beta_0(\nu^\sharp_0)$.  Hence the induced map in homology
$\Z\to H_0(J)$ is injective, but also surjective since
$\on{Sk}^i(x)\sim(-1)^{\length(x)}\beta_0(\nu^\sharp_0)$. Thus $H_0(J)=\Z$ in
this case.
\subsubsection{Computation of $H_l(J)$}
Suppose $\phi\in C_l(J)=\Z[V]$. Then $\partial\phi=0$ if and only if
$\on{Sk}^{0\cdots\hat{i}\cdots l}\phi=0$ for all $i$. That is, $\phi$
is invariant under every reflection $\sig_i\in\Waff$, hence under the
full affine Weyl group $\Waff$. But since $\phi$ has finite length
this is impossible unless $\phi=0$. This shows $H_l(J)=0$. 

\subsubsection{Computation of $H_p(J),\ 0<p<l$}
To simplify notation, we will write $C_\bullet$ instead of
$C_\bullet(J)$. (This should of course not be confused with the chain complex
$C_\bullet$ considered in previous sections.)
Introduce a $\Z$-filtration
\[ 0=F_{-1}C_\bullet\subset F_0C_\bullet\subset F_1C_\bullet\subset \cdots \]
where $F_NC_p$
is spanned by basis elements \eqref{eq:basis} with $|I|=p+1$ and 
$\length(x)\le N$. 
Formula \eqref{eq:ddd} shows
that for any 
basis element $\beta_I(x)\in F_NC_p$,  
\begin{equation}\label{eq:dd1}
 \partial \beta_I(x)={\sum_r}' (-1)^r \beta_{\delta_r I}(x)\mod
 F_{N-1}C_{p-1}\end{equation}
where the sum is only over those $r$ for which $x\in V^{\delta^r
  I}\subset V^I$, i.e. $u_r=1$ (other terms lower the filtration
degree since $\length(u_r x)< \length(x)$ unless $x=u_r x$).  In particular,
$\partial$ preserves the filtration. Define operators
$h_i:\,C_p\to C_{p+1}$ on basis elements, as follows:
\[ h_i\beta_I(x)= \begin{cases}
(-1)^r\beta_{I\cup\{i\}}(x)&\text{ if }\  
i_{r-1}<i<i_r,
\\
0&\text{ if $i=i_r$, some $r$. }\end{cases}
\]
Note that $h_i$ preserves the filtration: $h_i(F_NC_p)\subset F_NC_{p+1}$. 
%
%
Let
\[A_i=\on{id}-h_i\partial-\partial h_i.\] 
Then $A_i$ is a chain map, which is homotopic 
to the identity map. 
\begin{lemma}
Let $p>0$. 
For any basis element $\beta_I(x)\in F_NC_p$  we have 
$A_i\,\beta_{I}(x)\in F_{N-1}C_p$
\emph{unless}  $i\in I$ and $x\not\in V^{I-\{i\}}$. In the latter case, 
\[ A_i \beta_{I}(x)=
\beta_{I}(x)\mod  F_{N-1}C_p.\]
\end{lemma}
\begin{proof}
Write $I=\{i_0,\ldots,i_p\}$ where $i_0<\cdots<i_p$.  Using
\eqref{eq:dd1} we obtain
\begin{equation}\label{eq:subtle}
 h_i\partial \beta_I(x)={\sum_r}' (-1)^r h_i\beta_{\delta^r I}(x)\mod
 F_{N-1}C_p,\end{equation}
summing over indices with $x\in V^{\delta^r I}\subset V^I$. 
The calculation of $A_i\beta_I(x)$ divides into two cases: 
\noindent{\bf Case 1:} $i\in I$. Thus $i=i_{s}$ for some index $s$, and 
$(-1)^r h_i\beta_{\delta^r I}(x)=0$ unless $r=s$, in which case 
one obtains $\beta_I(x)$. Hence all terms in the sum \eqref{eq:subtle} 
vanish, except possibly for the term $r=s$ which appears if and only if 
$x\in V^{\delta^s I}=V^{I-\{i\}}$. That is, 
\[ h_i\partial\beta_I(x)=\begin{cases}\beta_I(x)\mod F_{N-1}C_p &\text{ if }x\in V^{I-\{i\}}\\
0 \mod F_{N-1}C_p &\text{ if }x\not\in V^{I-\{i\}}\end{cases}\]
(using the assumption $p>0$). Since $h_i\beta_I(x)=0$ this shows 
$A_i\beta_I(x)\in F_{N-1}C_p $ \emph{unless} $x\not \in V^{I-\{i\}}$, in which case $A_i \beta_I(x)=\beta_I(x) \mod F_{N-1}C_p$. 

\noindent {\bf Case 2:} $i\not\in I$. 
Exactly one of the terms 
in $\partial h_i \beta_I(x)$ reproduces $\beta_I(x)$. The remaining terms 
are organized in a sum similar to  \eqref{eq:ddd}: 
\[ \partial h_i \beta_I(x)=\beta_I(x)-{\sum_r}''(-1)^r
h_i\beta_{\delta^r I}(x) \mod F_{N-1}C_p,\]
where the sum is over all $r$ such that $x\in V^{I\cup\{i\}-\{i_r\}}$.
But $x\in V^{\delta^r I} \Leftrightarrow 
x\in V^{I\cup\{i\}-\{i_r\}}$, since
\[ V^{\delta^r I} =V^{I\cup\{i\}-\{i_r\}}\cap V^I .\]
Hence the sum $\sum_r'$ and $\sum_r''$ are just the same. This proves  
$A_i\,\beta_{I}(x)\in F_{N-1}C_p$. 
\end{proof}

Consider now the product $A:=A_0\cdots A_l$. 
By iterated application of the Lemma, we find that if 
$0<p<l$, then $A\beta_I(x)\in F_{N-1}C_p$
(because at least one index $i$ is not in $I$). Thus 
\[ A\colon F_NC_p\to F_{N-1}C_p\]
for $0<p<l$. 
The chain map $A$ is chain homotopic to the identity, since each of 
its factors are. Thus, if $\phi\in F_NC_p$ is a cycle,  
\[\phi\sim A\phi\sim\cdots  A^{N+1}\phi=0.\] 
This proves $H_p(J)=0$ for $0<p<l$, and concludes the proof of Theorem
\ref{th:combinatorial}.

\begin{remark}
  N. Kitchloo pointed out a more elegant proof of Theorem
  \ref{th:combinatorial}, along the lines of Kitchloo-Morava
  \cite{kit:th}.  His argument produces an inclusion of $C_\bullet(J)$
  as a direct summand of $S_\bullet\otimes_{\Z[W_J]}
  \Z$, where $S_\bullet$ is the simplicial complex with respect to the
  Stiefel diagram, and $\Z[W_J]$ acts on $\Z$ by the sign
  representation.  The acyclicity of $C_\bullet(J)$ then follows from
  the $W_J$-equivariant acyclicity of $S_\bullet$.
\end{remark}

\begin{appendix}
\section{Morita isomorphisms and  stable isomorphisms}
Let $\mathbb{H}_G$ be the stable $G$-Hilbert space, i.e. the 
unique (up to isomorphism) $G$-Hilbert space
containing all finite-dimensional unitary $G$-representations with
infinite multiplicity. (As a model, one may take for instance
$\mathbb{H}_G=L^2(G)\otimes L^2(\R)$.) Given a closed subgroup $H\subset G$,
the space $\mathbb{H}_G$ with the restricted action is a model for
$\mathbb{H}_H$. If $\ca{H}$ is any $G$-Hilbert space, its stabilization
$\H^{\on{st}}=\H\otimes\mathbb{H}_G$ is isomorphic to $\mathbb{H}_G$.  This
generalizes to bundles:
\begin{lemma}
  For any $G$-Hilbert bundle $\E\to X$, the stabilization
  $\E^{\on{st}}=\E\otimes \mathbb{H}_G$ is equivariantly isomorphic to
  the trivial bundle $X\times \mathbb{H}_G$, with diagonal $G$-action.
  Moreover, the isomorphism is unique up to a $G$-equivariant
  homotopy.
\end{lemma}
\begin{proof}
  Using induction over the cells of a $G$-CW decomposition of $X$, it
  is enough to show that for any $G$-Dixmier-Douady bundle $\E\to
  G/H\times \on{D}^r$, a given $G$-equivariant trivialization over $
  G/H\times \partial \on{D}^r$ extends to all of $G/H\times \on{D}^r$,
  and that the extension is unique up to homotopy. Since $G\times_H
  \mathbb{H}_G\cong G/H\times \mathbb{H}_G$, it suffices to extend the
  trivialization over $eH\times \partial D^r$ to an $H$-equivariant
  isomorphism between $\E$ and $\mathbb{H}_G$ over $eH\times
  \on{D}^r$. By the contractibility of the unitary group in the strong
  operator topology, and since the $H$-action on the base is trivial,
  such an extension exists and is unique up to homotopy.
\end{proof}

\begin{lemma}
For any $G$-Dixmier-Douady bundle $\A$, there is a canonical
equivariant Morita isomorphism $(\E,\psi)\colon \A\simeq
\A^{\on{st}}$.
\end{lemma}
\begin{proof}
Letting $\A_{HS}$ be the Hilbert bundle implementing the Morita isomorphism 
$\A\simeq\A$. Then  $\A_{HS}^{\on{st}}=\A_{HS}\otimes \mathbb{H}_G$ 
defines a Morita isomorphism $\A\simeq \A^{\on{st}}=\A\otimes \K_G$.  
\end{proof}
The following Lemma is a very simple special case of the
Brown-Green-Rieffel Theorem. 
\begin{lemma}
  For any two $G$-Dixmier-Douady bundles $\A,\B$, there is a 1-1
  correspondence between equivalence classes of equivariant stable isomorphisms
  $\A^{\on{st}}\to \B^{\on{st}}$ and equivalence classes of equivariant Morita
  isomorphisms $\A\simeq \B$.
\end{lemma}
\begin{proof}
Suppose $(\E,\psi)\colon \A\simeq \B$ is a $G$-Morita isomorphism. That
is, $\psi\colon \K(\E)\to \ca{B}\otimes\A^{\on{opp}}$ is a
$G$-equivariant isomorphism.  Tensoring with $\A$, and using the isomorphism
$\A\otimes\A^{\on{opp}}\cong \K(\A_{HS})$, we obtain an isomorphism
$\A\otimes \K(\E)\to \B\otimes \K(\A_{HS})$. Stabilize by tensoring
with $\K_G$. The choice of isomorphisms $\E\otimes
\mathbb{H}_G=\mathbb{H}_G$ and $\A_{HS}\otimes
\mathbb{H}_G=\mathbb{H}_G$ produces an isomorphism $\A\otimes \K_G
\xra{\cong} \B\otimes\K_G$, hence a $G$-equivariant isomorphism
\[ \A^{\on{st}}\xra{\cong} \B^{\on{st}}.\]
Conversely, given a stable equivariant isomorphism, we obtain a Morita
isomorphism by composition,
\[ \A\simeq \A^{\on{st}}\cong \B^{\on{st}}\simeq \B.\]
It is easily checked that this construction gives a bijection on
equivalence classes.
\end{proof}

\section{Relative Dixmier-Douady bundles}\label{sec:relative}
For any map $f\colon Y\to X$, and $\cone(f)$ its mapping cone,
obtained by gluing $\cone(Y)=Y\times I/Y\times\{0\}$ with $X$ by the
identification $(y,1)\sim f(y)$. Let
$H^\bullet(f)=H^\bullet(\cone(f))$ denote the relative cohomology of
$f$. Equivalently $H^\bullet(f)$ is the cohomology of the algebraic
mapping cone $C^\bullet(f)$ of the cochain map $C^\bullet(Y)\to
C^\bullet(X)$, i.e. $C^p(f)=C^{p-1}(Y)\oplus C^p(X)$ with differential
$\d(a,b)=(\d\,a-f^*b,\,\d c)$. If $f$ is a smooth map of manifolds,
the cohomology $H^\bullet(f,\R)$ may be computed using differential
forms, replacing the singular cochains in the above.

The group $H^2(f)$ has a geometric interpretation as isomorphism
classes of relative line bundles, i.e. pairs $(L,\psi_Y)$, where $L$
is a Hermitian line bundle over $X$, and $\psi_Y\colon Y\times\C\to
f^*L$ is a unitary trivialization of its pull-back to $Y$. (The class
of a relative line bundle is the Chern class of the line bundle
$\ti{L}\to \cone(f)$, obtained by gluing $\cone(Y)\times \C$ with $L$
via $\psi_Y$.

Similarly, $H^3(f)$ is interpreted in terms of relative Dixmier-Douady
bundles, i.e. triples $(\A,\E_Y,\psi_Y)$, where $\A\to X$ is a
Dixmier-Douady bundle, $\E_Y\to Y$ is a Hilbert space bundle, and
$\psi_Y\colon \K(\E_Y)\to \psi_Y^* \A$ a Morita trivialization of its
pull-back to $Y$. Given such a triple, one may construct a
Dixmier-Douady bundle $\ti{\A}\to \cone(f)$: First, use the Morita
trvialization $\psi_Y$ to define a stable trivialization $Y\times
\mathbb{K} \cong \K(\E_Y^{\on{st}}) \cong f^*\A^{\on{st}}$, then define
$\ti{\A}$ by gluing the trivial bundle $\cone(Y)\times\mathbb{K}$ with
$\A^{\on{st}}$ using this identification. We define the relative
Dixmier-Douady class $\on{DD}(\A,\E_Y,\psi_Y):=\on{DD}(\ti{A})$.

Tensor products and opposites of relative Dixmier-Douady bundles are
defines in the obvious way. A Morita trivialization of
$(\A,\E_Y,\psi_Y)$ is a triple $(\E_X,\psi_X,h_Y)$, consisting of a
Morita trivialization $\psi_X\colon \K(\E_X)\to \A$ and an isomorphism
$h_Y\colon \E_Y\to f^*\E_X$ intertwining $\psi_Y$ and $f^*\psi_X$. A
Morita isomorphism between two triples is a Morita trivialization of
their quotient. From the usual Dixmier-Douady theorem, one deduces
that $\on{DD}(\A,\E_Y,\psi_Y)$ classifies the Morita isomorphism
classes. 

More generally, one may define relative equivariant Dixmier-Douady
bundles; these are classified by an equivariant class
$\on{DD}_G(\A,\E_Y,\psi_Y) \in H^3_G(f):=H^3(f_G)$. Here 
$f_G\colon Y_G\to X_G$ is the induced map of Borel constructions.

\section{Review of Kasparov $K$-homology}
%
%
In this Section we review Kasparov's definition of K-homology
\cite{ka:eq,ka:to} for $C^*$-algebras. Excellent references for
this material are the books by Higson-Roe \cite{hig:ana} and Blackadar
\cite{bla:k}.
Suppose $\sA$ is a $\Z_2$-graded $C^*$-algebra, equipped with an action
of a compact Lie group $G$ by automorphisms. An equivariant Fredholm
module over $\sA$ is a triple $x=(\H,\varrho,F)$, where $\H$ is a
$G$-equivariant $\Z_2$-graded Hilbert space, $\varrho\colon \sA\to
L(\H)$ is a morphism of $\Z_2$-graded $G$-$C^*$-algebras, and $F\in
L(\H)$ is a $G$-invariant odd operator such that for all $a\in \sA$,
\[ (F^2-I)\varrho(a)\sim 0,\ (F^*-F)\varrho(a)\sim 0,\ [F,\varrho(a)]\sim 0.\]
Here $\sim$ denotes equality modulo compact operators. There is an
obvious notion of direct sum of Fredholm modules over $\sA$. One
defines a semi-group $K_0^G(\sA)$, with generators $[x]$ for each
Fredholm module over $\sA$, and equivalence relations
\[ [x]+[x']=[x\oplus x'],\] 
and 
\[ [x_0]=[x_1]\] 
provided $x_0,x_1$ are related by an `operator homotopy'
$x_t=(\H,\varrho,F_t)$ (cf.  \cite{bla:k,hig:ana}). One then proves
that every element in this semi-group has an additive inverse, so that
$K_0^G(\sA)$ is actually a group. More generally, for $q\le 0$ one
defines $K_q^G(\sA)=K_0^G(\sA\otimes \Cl(\R^q))$. 
This has the mod 2 periodicity property
$K^{q+2}_G(\sA)=K^q_G(\sA)$, which is then used to extend the
definition to all $q\in \Z$. The assignment $\sA\to K^q_G(\sA)$ is
homotopy invariant, contravariant functor, depending only on the
Morita equivalence class of $\sA$. It has the \emph{stability
property}, $K^q_G(\sA\otimes\mathbb{K}_G)=K^q_G(\sA)$. 

With this definition, let us now review some basic examples of twisted
$K$-homology groups $K_q^G(X,\A)=K^q_G(\Gamma_0(X,\A))$ for Dixmier-Douady bundles $\A\to X$. 

\begin{example}\label{ex:C1}
  Let $\A\to\pt$ be a $G$-equivariant Dixmier-Douady bundle over a
  point.  Thus $\A\cong \K(\E)$ for some Hilbert space $\E$.  As in
  Section \ref{subsec:trv} 
  the action $G\to \Aut(\A)$
  defines a central extension $\wh{G}$ of $G$ by $\U(1)$. The group
  $\wh{G}$ acts on $\E$, in such a way that the central circle acts
  with weight $1$. Let $V$ be a $\wh{G}$-module where the central
  circle acts with weight $-1$. Then the Hilbert space $\H=V\otimes\E$
  is a $G$-module. Letting $\rho\colon \C\to L(\H)$ be the action by
  scalar multiplication, the triple $(\H,\varrho,0)$ is a
  $G$-equivariant Fredholm module over $C(\pt)=\C$. This construction
  realizes the isomorphism $R(\wh{G})_{-1}\to K_0^G(\pt,\A)$.
\end{example}
\begin{example}\label{ex:c2}
  Let $M$ be a compact Riemannian $G$-manifold, and $D$ an invariant
  first order elliptic operator acting on a $G$-equivariant
  $\Z_2$-graded Hermitian vector bundle $\E=\E^+\oplus \E^-$.
  Suppose also that a finite rank $\Z_2$-graded $G$-Dixmier-Douady
  bundle $\A\to M$ acts on $\E$, where the action is equivariant and
  compatible with the grading.  Let $\H$ be the space of
  $L^2$-sections of $\E$, with the natural representation $\varrho$
  of $\Gamma(M,\A)$, and $F=D(1+D^2)^{-1/2}\in L(\H)$. The commutator
  of $F$ with elements $\varrho(a)$ for $a\in \Gamma(M,\A)$ are
  pseudo-differential operators of degree $-1$, hence are
  compact. Thus $(\H,\varrho,F)$ is an equivariant Fredholm module
  over $\Gamma(M,\A)$, defining a class in $K_0^G(M,\A)$.
\end{example}
\begin{example}\cite[page 114]{ka:con} 
  Let $M$ be a compact Riemannian $G$-manifold, and $\A=\Cl(TM)$ its
  Clifford bundle. Take $\E=\wedge T^*M$, $\H$ its space of
  $L^2$-sections, and $\varrho$ the usual action of sections of
  $\Gamma(M,\Cl(TM))$. Let $D=\d+\d^*$ be the de-Rham Dirac operator.
  By \ref{ex:c2} above, we obtain a Fredholm module $(\H,\varrho,F)$
  over $\Gamma(M,\Cl(TM))$, defining a class $[M]\in
  K_0^G(M,\Cl(TM))$. This is the \emph{Kasparov fundamental class} of
  $M$. (Actually, $\Cl(TM)$ is a Dixmier-Douady bundle only if $\dim
  M$ is even. If $\dim M$ is odd, one can use the isomorphism $
  K_0^G(M,\Cl(TM))=K_1^G(M,\Cl^+(TM))$ if needed. )
\end{example}
\begin{example}\label{ex:C4}
Let $H$ be a closed subgroup of $G$, and that $\B\to \pt$ is an
$H$-Dixmier-Douady bundle of finite rank. As explained in \ref{ex:C1},
any class in $K_0^H(\pt,\Cl(\g/\h)\otimes B)$ is realized by a
Fredholm module of the form $(\E,\varrho,0)$ where $\E$ if a Hilbert
space of finite dimension. Let $\wh{\E}=G\times_H \E$. The action of
$\Cl(T(G/H))$ defines a Dirac operator, which together with the action
of $\on{Ind}_H^G(\B)$ yields a Fredholm module and hence an element of
$K_0^G(G/H,\on{Ind}_H^G(\B))$.  This construction realizes the
isomorphism $K_0^H(\pt,\ca{B}\otimes\Cl(\g/\h)) \to
K_0^G(G/H,\on{ind}_H^G(\ca{B}))$ if $\B$ has finite rank.  Note that
since $H^3_H(\pt)$ is torsion, all $H$-Dixmier-Douady bundles over
$\pt$ are Morita isomorphic to finite rank ones.
\end{example}

%
%
%
%
%

\end{appendix}

\def\cprime{$'$} \def\polhk#1{\setbox0=\hbox{#1}{\ooalign{\hidewidth
  \lower1.5ex\hbox{`}\hidewidth\crcr\unhbox0}}} \def\cprime{$'$}
  \def\cprime{$'$} \def\polhk#1{\setbox0=\hbox{#1}{\ooalign{\hidewidth
  \lower1.5ex\hbox{`}\hidewidth\crcr\unhbox0}}} \def\cprime{$'$}
  \def\cprime{$'$}
\providecommand{\bysame}{\leavevmode\hbox to3em{\hrulefill}\thinspace}
\providecommand{\MR}{\relax\ifhmode\unskip\space\fi MR }
\providecommand{\MRhref}[2]{%
  \href{http://www.ams.org/mathscinet-getitem?mr=#1}{#2}
}
\providecommand{\href}[2]{#2}

\vskip1in

\end{document}